\documentclass[10pt,a4paper]{amsart}
\usepackage{amsaddr}
\usepackage{amssymb,amsthm,amsmath,amsfonts,amscd}
\usepackage{graphicx}
\usepackage{url}
\usepackage{cite}
\usepackage{hyperref}
\allowdisplaybreaks[1]
\usepackage[english]{babel}
\usepackage{enumerate}
\usepackage{bbm}
\usepackage{stmaryrd}
\usepackage{mathrsfs}

\newcommand{\norm}[1]{\left\lVert#1\right\rVert}

\newcommand{\abs}[1]{\lvert#1\rvert}
\newcommand{\lrabs}[1]{\left\lvert#1\right\rvert}

\newcommand{\lrbr}[2]{\left\langle#1,{#2}\right\rangle}

\newcommand{\dualdel}[3]{\sideset{_{#1^\ast}}{_{#1}}{\mathop{\left\langle{#2},{#3}\right\rangle}}}

\newcommand{\Ascr} {{\mathscr A}}
\newcommand{\Bscr} {{\mathscr B}}

\newcommand{\Escr} {{\mathscr E}}
\newcommand{\Fscr} {{\mathscr F}}

\newcommand{\Hscr} {{\mathscr H}}

\newcommand{\Wscr} {{\mathscr W}}

\renewcommand{\tilde}{\widetilde}
\renewcommand{\hat}{\widehat}
\renewcommand\a{\alpha}

\newcommand{\eps}{{\varepsilon}}
\renewcommand{\phi}{\varphi}

\renewcommand{\emptyset}{\varnothing}

\newcommand{\Z}{\mathbbm{Z}}
\newcommand{\R}{\mathbbm{R}}
\newcommand{\N}{\mathbbm{N}}

\renewcommand{\P}{\mathbbm{P}}
\newcommand{\E}{\mathbbm{E}}

\renewcommand{\div}{\operatorname{div}}

\newcommand{\Lip}{\operatorname{Lip}}

\newcommand{\rs}[1]{{\restriction}_{#1}}

\newcommand{\Id}{\operatorname{Id}}

\newcommand{\sgn}{\operatorname{sgn}}

\renewcommand{\limsup}{\varlimsup}
\renewcommand{\liminf}{\varliminf}

\newcommand{\BIGOP}[1]{\mathop{\mathchoice%
{\raise-0.22em\hbox{\huge $#1$}}%
{\raise-0.05em\hbox{\Large $#1$}}{\hbox{\large $#1$}}{#1}}}

\newcommand{\BIGboxplus}{\mathop{\mathchoice%
{\raise-0.35em\hbox{\huge $\boxplus$}}%
{\raise-0.15em\hbox{\Large $\boxplus$}}{\hbox{\large $\boxplus$}}{\boxplus}}}

\numberwithin{equation}{section}

\newtheorem{theorem}{Theorem}[section] 

\theoremstyle{plain}
\newtheorem{hypo}[theorem]{Hypothesis}
\newtheorem{defi}[theorem]{Definition}
\newtheorem{lem}[theorem]{Lemma}
\newtheorem{ex}[theorem]{Example}
\newtheorem{prop}[theorem]{Proposition}
\newtheorem{rem}[theorem]{Remark}

\newtheorem{thm}[theorem]{Theorem}

\newtheorem{cor}[theorem]{Corollary}
\newtheorem*{bem*}{Bemerkung}

\newcommand{\ol}{\overline}

\newcommand{\Sgn}{\operatorname{Sgn}}

\renewcommand{\subset}{\subseteq}
\renewcommand{\supset}{\supseteq}

\title{Multi-valued, singular stochastic evolution inclusions}
\author[B. Gess \& J. M. T\"olle]{Benjamin Gess and Jonas M. T\"olle}

\email{gess@math.tu-berlin.de}

\email{jonasmtoelle@gmail.com}
\address{Institut f\"ur Mathematik, Technische Universit\"at Berlin (MA 7-5)\\
Stra\ss{}e des 17. Juni 136, 10623 Berlin, Germany}
\thanks{{\bf Acknowledgements:} Financial support of the German Science Foundation (DFG), in particular of the IRTG 1132 ``Stochastics and Real World Models'', Beijing--Bielefeld and the Forschergruppe 718 ``Analysis and Stochastics in Complex Physical Systems'', Berlin--Leipzig is gratefully acknowledged.\\
The authors would like to thank Max von Renesse for inspiring discussions.}

\date{\today}

\begin{document}
\begin{abstract}
We provide an abstract variational existence and uniqueness result for multi-valued, monotone, non-coercive stochastic evolution inclusions in Hilbert spaces with general additive and Wiener multiplicative noise. As examples we discuss certain singular diffusion equations such as the stochastic $1$-Laplacian evolution (total variation flow) in all space dimensions and the stochastic singular fast diffusion equation. In case of additive Wiener noise we prove the existence of a unique weak-$*$ mean ergodic invariant measure.
\end{abstract}

\maketitle

\section{Introduction}

We consider the following evolution inclusion in a separable Hilbert space $H$
\begin{equation}\label{eqn:det:int}\begin{split}
  dX_t+A(t,X_t)\,dt&\ni\, dg_t,\quad t>0,\\
  X_0&=x.
\end{split}\end{equation}
Here $A$ is (among other assumptions) required to be a possibly multi-valued, singular, maximal monotone operator and $g$ is a c\`adl\`ag path in $H$. The meaning of the expression $dg_t$ will be specified below.

In particular, we are interested in stochastic evolution inclusions of the type
\begin{equation}\label{eqn:stoch_add}\begin{split}
  dX_t+A(t,X_t)\,dt&\ni dN_t,\quad t>0,\\
  X_0&=x,
\end{split}\end{equation}
where $\{N_t\}_{t\ge 0}$ is a c\`adl\`ag, adapted $H$-valued stochastic process on a filtered probability space $(\Omega,\Fscr,\{\Fscr_t\}_{t\ge 0},\P)$ and in inclusions of the form
\begin{equation}\label{eqn:stoch_mult}\begin{split}
  dX_t+A(t,X_t)\,dt&\ni B_t(X_t)\,dW_t,\quad t>0,\\
  X_0&=x,
\end{split}\end{equation}
where for some separable Hilbert space $U$, $B: [0,T] \times \Omega \times S \to L_2(U;H)$ takes values in the space of Hilbert-Schmidt operators from $U$ to $H$ and $\{W_t\}$ is a cylindrical Wiener process.

We prove the unique existence of solutions to \eqref{eqn:det:int}--\eqref{eqn:stoch_mult} as well as the unique existence of a weak-$*$ mean ergodic invariant measure for \eqref{eqn:stoch_mult} with additive noise, i.e.\ with $B_t(x) \equiv B_t$ independent of $x$.

We now comment on the main mathematical difficulties arising for singular evolution inclusions of type \eqref{eqn:det:int}--\eqref{eqn:stoch_mult}. The standard variational approach to (S)PDE of type \eqref{eqn:stoch_mult} requires the drift operator $A$ to be single-valued and to extend to a hemicontinuous, coercive operator $A: V \to V^*$ for some Gelfand triple $V \subseteq H \subseteq V^*$ (cf.\ \cite{KrRo0,Pard,PrRoe}). The reflexivity of $V$ and $V^*$ is crucial for the construction of solutions. Therefore, the standard approach cannot be applied to highly singular (S)PDE such as the total variation flow, the two phase Stefan problem, plasma diffusion and the curve shortening flow. In all of these examples the space $V$ degenerates in the sense that $V$ or $V^*$ fail to be reflexive. While recently increasing interest has been paid to this kind of singular, possibly multi-valued SPDE (cf.\ e.g.\ \cite{BDP3,BDPR,BenRas,Ciot3,Ras,Step2}), the unique existence of solutions could only be shown for additive noise and under strong dimensional restrictions. The principal idea of most of these works is the concept of (stochastic) 
evolution variational inequalities (EVI), thus weakening the notion of solutions to \eqref{eqn:stoch_mult}\footnote{EVIs are also referred to as SVIs, that is, stochastic variational inequalities. We are following this naming convention in Appendix \ref{SVIapp}.}. However, the approach via EVI has multiple drawbacks. First, it relies on the transformation of \eqref{eqn:stoch_mult} into a random PDE and hence is restricted to simple structures of noise, such as additive or linear multiplicative noise. Second, due to the weaker notion of solutions it is hard to prove uniqueness. In fact, so far uniqueness of EVI could only be proven in case of sufficiently regular additive noise. Third, the construction of solutions to EVI still requires a coercivity condition of the type
  \begin{equation}\label{coerc}\dualdel{V}{A(u)}{u}\ge c\norm{u}_V^\alpha.\end{equation}
for some $\alpha\ge 1$, $c>0$, which leads to restrictions on the dimension or the coercivity exponent $\alpha$. 

In order to remedy these obstacles for the class of equations considered in this paper, we introduce another Hilbert space $S$ embedded compactly and densely into $H$ such that
  \[S\subset V\subset H\equiv H^\ast\subset V^\ast\subset S^\ast.\]
Subsequently, we will drop the intermediate space $V$ and formulate the conditions of our hypotheses solely with respect to $S$. We assume that the drift $A$ is maximal monotone and of at most linear growth in $S^\ast$. We are able to replace the strong coercivity assumption \eqref{coerc} by weak dissipativity in $S$ formulated in an approximative way (cf.\ (A4) below). 

Once existence and uniqueness of solutions for additive noise has been shown, limiting solutions for multiplicative noise may be constructed by ``freezing the diffusion coefficients'' and iteration methods. This is a well-known technique relying on the monotonicity of the drift (cf.\ e.g.\ \cite{BDPR3,BDPR09}). However, the solutions constructed this way are in general not known to satisfy the actual SPDE in any way. By proving regularity properties (i.e.\ $X_t$ taking values in $S$ a.s.) we are able to identify the limiting solutions as variational solutions to \eqref{eqn:stoch_mult} (cf.\ Definition \ref{def:soln_mult} below).

Both existence and uniqueness of invariant measures for singular SPDE of the form \eqref{eqn:stoch_mult} with additive noise are difficult problems in general. Since only in dimension $d=1$ we expect a compact embedding of the energy space $V$ into $H$, for higher dimensions the standard approach to prove existence of invariant measures via the Krylov-Bogoliubov Theorem does not apply. Accordingly, for example, the existence of invariant measures for the stochastic total variation flow, i.e.
 $$ dX_t = \operatorname{div}\left[\abs{\nabla X_t}^{-1}\nabla X_t\right]\,dt+dB_t,$$
could so far only be shown in one dimension in \cite{BDPR}\footnote{In the recent work \cite{Kim12} the existence and uniqueness of an invariant measure for the total variation flow on $\R^d$ has been proven for all $d\in\N$. This work became available after the present paper was finished and poses the case of bounded domains as an open problem which will be solved here.}. Due to the lack of strong monotonicity of the drift operators $A$ considered in this paper, also the dissipativity approach to prove uniqueness and ergodicity of invariant measures (cf.\ e.g.\ \cite{BDP} and the references therein) cannot be used. In \cite{ESvR} a general abstract result from \cite{KPS} has been applied to prove uniqueness of the invariant measure for the stochastic curve shortening flow in one dimension. However, their proof still relies on the compactness of the embedding $V \subseteq H$ of the underlying Gelfand triple. By refining their analysis we manage to avoid the need of compactness and thus to prove existence and uniqueness of an invariant measure for the stochastic total variation flow also for $d = 2$. In case of the stochastic plasma diffusion equation (cf.\ \eqref{FD-ln} below) the situation is even more involved, since the homogeneity of the drift operator $A$ is lost. By deriving new, locally uniform decay estimates for its deterministic 
counterpart we manage to prove existence and uniqueness of invariant measures nonetheless. 

For additive noise, we choose a pathwise approach to construct the solutions, which allows to include general noise, such as L\'evy noise and fractional Brownian motion with arbitrary Hurst parameter. Since the literature on SPDE driven by L\'evy noise is extensive, we make no attempt to provide a complete list but only refer to \cite{G82,PZ07} and the references therein. Another advantage of a pathwise construction is that the existence of a stochastic flow and a random dynamical system (RDS) are immediate consequences. Thus, assuming that the noise $\{N_t\}$ in \eqref{eqn:stoch_add} has c\`adl\`ag paths in $S$, satisfies some spatial regularity and has strictly stationary increments we prove that there is an RDS associated to \eqref{eqn:stoch_add}. In case of additive Wiener noise this yields the Markov property for the associated semigroups $\{P_t\}$. 

Using stochastic calculus we prove that in case of multiplicative Wiener noise we can relax the spatial regularity assumptions on the noise while preserving the 
regularity of the solutions. For additive Wiener noise, we prove the existence and uniqueness of a weak-$*$ mean ergodic measure using recent methods by Komorowski, Peszat and Szarek \cite{KPS}.

Let us now describe the concrete applications covered by our general results in more detail. As a first example we include certain singular diffusion equations such as the stochastic $p$-Laplacian evolution, $p \in [1,2)$ (including the total variation flow, i.e.\ $p=1$, where
the equation becomes a multi-valued inclusion) in all space dimensions. More precisely,                                                                                                                                                 \begin{equation}\label{pLap}\begin{split}
  dX_t&\in\operatorname{div}\left[\abs{\nabla X_t}^{p-2}\nabla X_t\right]\,dt+\left\{\begin{aligned}&\,dN_t,\\&B_t(X_t)\,dW_t,\end{aligned}\right.\quad t>0,\\
  X_0&=x,
\end{split}\end{equation}
where we use the curly bracket to indicate that we either consider c\`adl\`ag additive noise $N_t$ as specified above or multiplicative Wiener noise. For all equations considered in this paper, unique existence of solutions is proven for general additive c\`adl\`ag noise and multiplicative Wiener noise in all space dimensions. As applied to \eqref{pLap} our results partially generalize the results of \cite{BDPR} (known to hold only in space dimensions $d=1,2$ and for Wiener additive noise) and \cite{BR12} (space dimensions $d=1,2,3$, linear multiplicative noise). In addition, we solve the open problem posed in \cite{BDPR} of uniqueness of the invariant measure for \eqref{pLap} as mentioned above. We note, however, that a key point of \cite{BDPR,BR12} is the well-posedness and characterization of solutions for initial conditions $x\in L^2(\Lambda)$ while we only construct solutions in a limiting sense for $x\in L^2(\Lambda)$ (using the well-posedness results from \cite{BDPR,BR12} these two notions of solutions coincide, cf.\ Appendix \ref{SVIapp} below). Equation \eqref{pLap} is a stochastically perturbed version of the deterministic $p$-Laplace equation which in its deterministic form appears in geometry, quasi-regular mappings, non-Newtonian fluid dynamics, plasma physics and image restoration (cf.\ e.g.\ \cite{Lad,ROF92}).

A further application of our general existence and uniqueness results, is the (1+1)-dim. stochastic curve shortening flow:
\begin{equation}\begin{split}
  dX_t(\xi)&=\frac{\partial_\xi^2 X_t(\xi)}{1+(\partial_\xi X_t(\xi))^2}\,dt+\left\{\begin{aligned}&\,dN_t,\\&B_t(X_t)\,dW_t,\end{aligned}\right.\quad\xi\in [0,1],\quad t>0,\\
  X_0(\xi)&=x(\xi),\end{split}
\end{equation}
which has been studied by Es-Sarhir and von Renesse \cite{ESvR} (see also \cite{ESvRS}). We are able to generalize their result about unique existence of solutions to the case of general additive noise.

More generally, we study the existence and uniqueness of solutions, as well as ergodicity (for additive Wiener noise) for
generalized $\Phi$-Laplacian equations with Neumann (zero average for ergodicity) or Dirichlet boundary conditions on bounded domains $\Lambda\subset\R^d$, $d\in\N$:
\begin{equation}\begin{split}dX_t&\in \div \Phi(\nabla X_t)\,dt+\left\{\begin{aligned}&\,dN_t,\\&B_t(X_t)\,dW_t,\end{aligned}\right.\quad t> 0,\\
X_0&=x,\end{split}\end{equation}
where $\Phi:\R^d\to 2^{\R^d}$ is given as the subdifferential of a convex, continuous function of subquadratic growth and $\Phi(0)=0$. 

Further admissible nonlinearities $\Phi$ include the ones appearing in the
evolution problem for plastic antiplanar shear (cf.\ \cite{Zhou}) and the so-called minimal surface flow (cf.\ \cite{LichTem}).

We also consider the stochastic singular fast diffusion equation for $r\in [0,1)$ (again including the critical, multi-valued case $r = 0$):
\begin{equation}\label{FD}\begin{split}
dX_t&\in\Delta\left[\abs{X_t}^{r-1}X_t\right]\,dt+\left\{\begin{aligned}&\,dN_t,\\&B_t(X_t)\,dW_t,\end{aligned}\right.\quad t>0,\\
X_0&=x.
\end{split}\end{equation}
This generalizes results given in \cite{BDP2} where only additive case has been considered. Moreover, we solve the open problem posed in \cite{BDP2,BDPR12} of uniqueness of the invariant measure for \eqref{FD} in one space dimension. Equation \eqref{FD} is a randomly perturbed version of the deterministic fast diffusion equation which, for example, arises in plasma physics (cf. \cite{Vaz2} and the references therein). In its stochastically perturbed form \eqref{FD} appears in the study of self-organized criticality (cf.\ \cite{BDPR2} and references therein). In models of self-organized criticality, the spontaneous, stable approach of a critical state, i.e.\ ergodicity and mixing properties are key features. Under dimensional restrictions, Liu and the second named author proved ergodicity and polynomial decay for the 
equations \eqref{pLap} and \eqref{FD}, excluding the limit cases $p=1$, $r=0$ in \cite{LiuToe1}.

Stochastic porous media equations and stochastic fast diffusion equations have been intensively investigated in recent years (cf.\ \cite{DPRRW06,RRW,BDPR09} and references therein). The stochastic fast diffusion equation ($0 < r < 1$) with linear multiplicative space-time noise has been first solved in \cite{RRW}. Subsequently, extinction in finite time with positive probability has been shown in \cite{BDPR2,BDPR09-3}.

We are also able to treat the so called stochastic diffusion equation of plasma for any space dimension:
\begin{equation}\label{FD-ln}\begin{split}
dX_t&=\Delta\left[\log\left(\abs{X_t}+1\right)\sgn(X_t)\right]\,dt+\left\{\begin{aligned}&\,dN_t,\\&B_t(X_t)\,dW_t,\end{aligned}\right.\quad t>0,\\
X_0&=x,
\end{split}\end{equation}
previously studied in an EVI setting in (1+1)-dimensions by Ciotir \cite{Ciot3}. Equation \eqref{FD-ln} is the stochastically perturbed version of a PDE arising in plasma physics, gas kinetics and statistical mechanics (cf.\  \cite{BerryHolland82}). For additive Wiener noise, we are able to
prove the existence of a unique invariant measure in one and two space dimensions.

For general symmetric negative-definite Dirichlet operators $L$ on $L^2(E,\Bscr,\mu)$, that
satisfy a spectral gap condition (here $(E,\Bscr,\mu)$ is an abstract separable measure space), we consider the so called generalized fast-diffusion equation
\begin{equation}\begin{split}dX_t&\in L\left[\phi({X_t})\right]\,dt+\left\{\begin{aligned}&\,dN_t,\\&B_t(X_t)\,dW_t,\end{aligned}\right.\quad t> 0,\\
X_0&=x,\end{split}\end{equation}
where $\phi:\R\to 2^\R$ is a monotone graph ($0\in\phi(0)$) with sublinear growth.

Future applications to RDS and random attractors similar to \cite{BGLR,Gess1,GessLiuRoe} as well as continuity results in the parameter $p$ in \eqref{pLap} ($r$ in \eqref{FD} resp.), similar to
\cite{CiotToe,CiotToe2}, are in preparation.

\section{Deterministic case}\label{sec:det_case}

We start by considering the deterministic, time-inhomogeneous evolution inclusion \eqref{eqn:det:int}. Let $H$ be a separable Hilbert space with dual $H^\ast$. 
Suppose that there is another Hilbert space $S$ embedded densely and compactly into $H$. We thus have a Gelfand triple
  \[S\subset H \subset S^\ast\]
and it holds that
  \[\dualdel{S}{v}{u}=(v,u)_H, \quad\text{whenever } u\in S,\;v\in H.\]
Let $i_S:S\to S^\ast$ denote the Riesz map of $S$. We note that the scalar product $(\cdot,\cdot)_S$ defines a bilinear, $S$-bounded, $S$-coercive form on $H$. By the Lax-Milgram Theorem there is a linear, positive definite, self-adjoint operator $T: D(T) \subseteq H \to H$ with $D(T^{1/2}) = S$ and $(T^{1/2}u,T^{1/2}v)_H = (u,v)_S$. We define $J_n = (1+\frac{T}{n})^{-1}$, $n\in\N$, to be the resolvent associated to $T$ and $T_n = TJ_n = n(1-J_n)$ to be the Yosida approximation of $T$. Then
  \[(x,y)_n:=(x,T_n y)_H,\quad x,y\in H,\]
form a sequence of new inner products on $H$, the induced norms $\norm{\cdot}_n$ are all equivalent to $\norm{\cdot}_H$ and
  \[\forall x\in S\; : \quad \norm{x}_n\uparrow\norm{x}_S\quad\text{as $n\to\infty$}.\]
Let $H_n:=(H,(\cdot,\cdot)_n)$. We will use the norms $\|\cdot\|_{n}$ as smooth approximations of the stronger norm $\|\cdot\|_S$. Due to the equivalence of $\|\cdot\|_{n}$ to $\|\cdot\|_{H}$ we may apply It\={o}'s formula to $\|\cdot\|_{n}^2$ and then pass to the limit $n \to \infty$. This will allow to derive estimates on the $S$-norm of approximating solutions to  \eqref{eqn:det:int}, see Lemma \ref{boundlem} below.

Let $g: [0,T] \to S$ be a c\`adl\`ag function (for simplicity $g(0)=0$). We consider the integral evolution inclusion
\begin{equation}\label{eqn:det}
   X_t \in x - \int_0^t A(s,X_s)\,ds+ g_t, 
\end{equation}
where $A:[0,T] \times S \to 2^{S^\ast}$ is a multi-valued operator. We will construct solutions to this equation by first using the transformation $Y_t = X_t - g_t$ which leads to the evolution inclusion
\begin{equation}\label{eqn:trans_det}\left\{\begin{aligned}
    dY_t+A(t,Y_t+g_t)\,dt&\ni 0,\quad t>0,\\
    Y_0&=x.
\end{aligned}\right.\end{equation}

We will prove unique existence for \eqref{eqn:trans_det} which leads to unique existence for \eqref{eqn:det} by transforming back, i.e.\ by defining $X_t := Y_t + g_t$.

Let $W^{1,2}(0,T)$ be defined by
\[W^{1,2}(0,T):=\left\{u\in L^2([0,T];S)\;\big\vert\;u\;\text{abs. cont.},\;\tfrac{d}{dt}u\in L^2([0,T];S^\ast)\right\},\]
cf.\ e.g.\ \cite[\S III.1]{Show} and let $D([0,T];H)$ be the space of all c\`adl\`ag functions in $H$ endowed with the topology of \textit{uniform convergence}. A solution of \eqref{eqn:trans_det} is a function $Y\in W^{1,2}(0,T)$ such that
    \[\frac{d}{dt}Y_t = -\zeta_t,\quad \text{ for a.e. } t \in [0,T],\]
$Y_0=x$ and $\zeta \in L^2([0,T];S^\ast)$ such
that $\zeta_t\in A(t,Y_t + g_t)$ for a.e. $t\in [0,T]$.
   
\begin{defi}\label{def:det}
  A \emph{solution} of \eqref{eqn:det} is a c\`adl\`ag function $X \in D([0,T];H)$ such that
    \[ X_t = x - \int_0^t \eta_s \,ds + g_t,\]
  for all $t \in [0,T]$ as an equation in $S^*$, where $\eta \in L^2([0,T];S^\ast)$ such that $\eta_t \in A(t,X_t)$ for a.e. $t\in [0,T]$.
\end{defi}

\begin{defi}\label{defi:measurable}
  Let $(E,\Bscr,\mu)$ be a $\sigma$-finite complete measure space and $Y$ be a Polish space. A map $F:E \to 2^Y$ with non-empty, closed values is called $\Bscr$-\emph{measurable} or simply \emph{measurable} if $\{F(\cdot)\cap O\not=\emptyset\}\in\Bscr$ for each open set $O\subset Y$.
\end{defi}

\begin{hypo}\label{hyp:A}
  Suppose that $A:[0,T] \times S \to 2^{S^\ast}$ satisfies the following conditions: There is a constant $C > 0 $ such that 
    \begin{enumerate}
  \item[\textup{(A1)}] \emph{Maximal monotonicity:} For all $t\in [0,T]$, the map $x\mapsto A(t,x)$ is maximal monotone with nonempty values.
  \item[\textup{(A2)}] \emph{Linear growth:} For all $x\in S$, for Lebesgue a.a.\ $t\in [0,T]$,
  for all $y\in A(t,x)$:
     \[\norm{y}_{S^\ast}\le f(t) + C\norm{x}_S,\]
     with some $f \in L^2([0,T])$.
  \item[\textup{(A3)}] \emph{Weak coercivity in $S$:} For all $x\in S$, for Lebesgue a.a. $t\in [0,T]$,
  for all $y\in A(t,x)$, and for all $ n\in\N$:
      $$\label{tn1} 2\dualdel{S}{y}{T_n (x)}\ge - \gamma(t) - C\norm{x}_S^2,$$ 
    with some $\gamma \in L^1([0,T])$, with $C,\gamma$ independent of $n$.
  \item[\textup{(A4)}] \emph{Measurability:}
      The map $t\mapsto A(t,x)$ is measurable w.r.t.\ the Lebesgue $\sigma$-algebra for all $x\in S$.
\end{enumerate}
\end{hypo}

A condition similar to (A3) has been used in \cite{Liu3,RW}. Note that (A4) can be dropped if $A$ is independent of time.

\begin{rem}
Using a standard transformation, it is sufficient for the unique existence of solutions to \eqref{eqn:det} to require that Hypothesis \ref{hyp:A} is satisfied for $A+\lambda \Id_H$ for some $\lambda > 0$. In other words, we may allow $A$ to be \emph{maximal quasi-monotone} with nonempty values.
\end{rem}

\begin{hypo}\label{hyp:g} 
   $g \in L^2([0,T];D(T^{3/2}))$.
\end{hypo}
Here, the operator $(T^{3/2},D(T^{3/2}))$ is defined in terms of the spectral theorem and the Hilbert space
$D(T^{3/2})$ is equipped with the graph norm.

In order to construct solutions to \eqref{eqn:trans_det} we will consider a viscosity approximation. Let $\eps>0$ and consider the perturbed problem
\begin{equation}\label{pereq}\left\{\begin{aligned}
    \frac{d}{dt}Y_t^\eps + \eps i_S(Y_t^\eps) &\in -A(t,Y_t^\eps+g_t),\quad t>0,\\
    Y_0^\eps&=x,
\end{aligned}\right.\end{equation}
The unique existence of variational solutions to these approximating problems is proved in Proposition \ref{prop:unique_ex_approx} below. Application of the transformation $X^\varepsilon_t:=Y^\varepsilon_t+g_t$ for the approximating equation yields
\begin{equation}\label{eqn:approx_det}
  X_t^\eps \in x - \int_0^t A(s,X_s^\eps) - \eps i_S(X_s^\eps-g_s)\ ds + g_t  ,\quad t>0.
\end{equation}
Letting $\eps \to 0$ we will prove the existence of solutions to  \eqref{eqn:trans_det}. Defining $X_t := Y_t + g_t$ we will then obtain
  
\begin{thm}\label{thm:unique_ex_det_smooth}
Assume Hypotheses \ref{hyp:A} and \ref{hyp:g} and let $x \in S$. Then \eqref{eqn:det} has a unique solution in the sense of Definition \ref{def:det} satisfying $X \in L^\infty([0,T];S)$ with
  \[ \sup_{t \in [0,T]}\|X_t^\eps-X_t\|_H^2+ \int_0^T\norm{X^\eps_t-X_t}_S^2\,dt\to 0\]
and $X$ is right-continuous in $S$. 
\end{thm}
\begin{proof}
See section \ref{proof:unique_ex_det_smooth}.
\end{proof}

By monotonicity of the drift we can extend the unique existence of solutions to all initial conditions $x \in H$ at the cost of passing to limiting solutions in the sense
\begin{defi}[Limit solution]\label{def:det_limit}
  We say that a function $X \in D([0,T];H)$ is a \emph{limit solution} to \eqref{eqn:det} with starting point $x \in H$, if $X_0=x$ and for each approximation $\{x^\delta\}$ with (eventually) $x^\delta \in S$ and $x^\delta \to x$ in $H$ the associated solutions $\{X^\delta\}$ converge to $X$ in $D([0,T];H)$.
\end{defi}

We obtain
\begin{thm}[Extension to all initial conditions $x \in H$]\label{thm:unique_ex_det_all}
  Suppose that Hypotheses \ref{hyp:A} and \ref{hyp:g} hold and let $x \in H$. Then there is a unique limit solution $X$.
\end{thm}

\section{Stochastic evolution inclusions with additive noise}\label{sec:add_noise}

We now consider stochastic evolution inclusions with additive noise of the type \eqref{eqn:stoch_add}. Let $(\Omega,\Fscr,\{\Fscr_t\}_{t\ge 0},\P)$ be a filtered probability space (not necessarily complete, nor right-continuous), $N: [0,T] \times \Omega \to S$ be an $\{\Fscr_t\}_{t\ge 0}$-adapted stochastic process with c\`adl\`ag paths in $S$ and $N_0=0$. We define $L^0(\Omega,\Fscr_0;H)$ to be the space of all $\Fscr_0$-measurable, $H$-valued random variables and let $x \in L^0(\Omega,\Fscr_0;H)$. For each $\omega\in\Omega$ we consider the following integral equation in $S^\ast$
\begin{equation}\label{aeq}
   X_t(\omega) \in x(\omega)-\int_0^t A(s,X_s(\omega))\,ds + N_t(\omega).
\end{equation}

\begin{defi}\label{def:stoch_soln_add_pathwise}
  An $\{\Fscr_t\}_{t \in [0,T]}$-adapted stochastic process $X:[0,T]\times\Omega\to H$ is a \emph{pathwise (limit) solution} to \eqref{aeq} with starting point $x \in L^0(\Omega,\Fscr_0;H)$ if for all $\omega \in \Omega$, $X(\omega)$ is a (limit) solution for \eqref{aeq} with $g_\cdot = N_\cdot(\omega)$.
\end{defi}

Setting $g_t := N_t(\omega)$ for fixed $\omega \in \Omega$, Theorem \ref{thm:unique_ex_det_smooth} and Theorem \ref{thm:unique_ex_det_all} yield the existence of a pathwise (limit) solution $X$ as long as $A$ satisfies Hypothesis \ref{hyp:A} and $N_\cdot(\omega)$ satisfies Hypothesis \ref{hyp:g}, i.e.\ $N_\cdot(\omega) \in L^2([0,T];D(T^{3/2}))$ for each $\omega \in \Omega$. The $\{\Fscr_t\}_{t \in [0,T]}$-adaptedness of $X$ is proved in Section \ref{sec:pathwise_additive_noise} below. We obtain

\begin{theorem}\label{thm:unique_ex_stoch_add_pathw}
  Assume that $A$ satisfies Hypothesis \ref{hyp:A}
and for each $\omega\in\Omega$ the path $g_\cdot=N_\cdot(\omega)$ satisfies Hypothesis \ref{hyp:g}. For $x \in L^0(\Omega,\Fscr_0;S)$ there is a unique pathwise solution to \eqref{aeq} in the sense of Definition \ref{def:stoch_soln_add_pathwise} satisfying $X(\omega) \in L^\infty([0,T];S)$ and
    \[\sup_{t \in [0,T]}\|X_t^\eps(\omega)-X_t(\omega)\|_H^2 + \int_0^T\norm{X^\eps_t(\omega)-X_t(\omega)}_S^2\,dt \to 0, \quad \forall \omega \in \Omega.\]
  Moreover, $X(\omega)$ is right-continuous in $S$. For $x \in L^0(\Omega,\Fscr_0;H)$ there is a unique pathwise limit solution to \eqref{aeq}. 
\end{theorem}
If the noise is two-sided and strictly stationary then the solutions generate a random dynamical system (RDS). Let $((\Omega,\Fscr,\P),(\theta_t)_{t \in \R})$ be a metric dynamical system, i.e. $(t,\omega) \mapsto \theta_t(\omega)$ is $(\mathscr{B}(\R) \otimes \Fscr,\Fscr)$ measurable, $\theta_0 =\Id$, $\theta_{t+s} = \theta_t \circ \theta_s$ and $\theta_t$ is $\P$-preserving, for all $s,t \in \R$. We assume that $N: \R \times \Omega \to S$ satisfies
\begin{hypo}\label{hyp:N}
  \item[\textup{(N)}] For all $t \ge s$ and $\omega \in \Omega$
    \[ N_t(\omega) - N_s(\omega) = N_{t-s}(\theta_s\omega). \]
\end{hypo}
By \cite[Lemma 3.1]{GessLiuRoe} for each $S$ valued process $\tilde N_t$ with $\tilde N_0 = 0$ a.s., stationary increments and a.s.\ c\`adl\`ag paths there exists a metric dynamical system $((\Omega,\Fscr,\P),(\theta_t)_{t \in \R})$ and a version $N_t$ of $\tilde N_t$ on $((\Omega,\Fscr,\P),(\theta_t)_{t \in \R})$ such that $N_t$ satisfies $(N)$. In particular, applications include all L\'evy processes and fractional Brownian motion with arbitrary Hurst parameter.

Since we constructed the solution for each path, we obtain
\begin{cor}\label{cor:gen_RDS}
  Assume that $A$ satisfies Hypothesis \ref{hyp:A}, is independent of time $t$ and $N$ satisfies Hypothesis \ref{hyp:g} and \ref{hyp:N} for all $\omega \in \Omega$. Then
    \[ \phi(t,\omega)x := X_t^x(\omega),\quad x \in H,\ t \in \R_+,\ \omega \in \Omega \]
  defines a continuous RDS associated to \eqref{eqn:stoch_add}, where $X_t^x(\omega)$ is the pathwise limit solution starting at $x$ obtained in Theorem \ref{thm:unique_ex_stoch_add_pathw}.
\end{cor}

\section{Stochastic evolution inclusions with multiplicative noise}

In this section we consider stochastic evolution inclusions driven by multiplicative Wiener noise, i.e.\ inclusions of the form \eqref{eqn:stoch_mult}. In this case, we may apply stochastic calculus and in particular It\=o's formula to strengthen our results in two directions: First, we prove that the regularity assumptions on the noise (Hypothesis \ref{hyp:g}) may be relaxed (cf.\ Hypothesis \ref{hyp:noise_2} below) while preserving the regularity of the solution. Second, we allow multiplicative noise, i.e.\ state-dependent diffusion coefficients $B_t$.

In order to be able to use stochastic calculus we require a normal filtered probability space $(\Omega,\hat\Fscr,\{\hat\Fscr_t\}_{t\ge 0},\P)$ with a cylindrical Wiener process $\{W_t\}_{t\ge 0}$ in $U$, where $U$ is some separable Hilbert space. We further require the diffusion coefficients $B: [0,T] \times \Omega \times S \to L_2(U,H)$ to be progressively measurable (i.e.\ for every $t \in [0,T]$ the map $B:[0,t] \times \Omega \times S\to L_2(U;H)$ is $\mathscr{B}([0,t])\otimes \hat\Fscr_t \otimes \mathscr{B}(S)$-measurable) and the following random version of Hypothesis \ref{hyp:g}:

\begin{hypo}\label{hyp:noise}
  \begin{enumerate}
   \item[\textup{(B1)}] There is an $h \in L^1([0,T]\times\Omega)$ such that
        \begin{equation*}\|B_t(x)\|_{L_2(U,H)}^2 \le C \|x\|_S^2 + h_t,\quad\textup{(growth)},\end{equation*}
        for all $t \in [0,T]$, $x\in S$ and $\omega \in \Omega$.
   \item[\textup{(B2)}] There is a $C > 0$ such that
        \begin{equation*}\|B_t(x)-B_t(y)\|_{L_2(U,H)}^2 \le C \|x-y\|_H^2,\quad\textup{(Lipschitz continuity)},\end{equation*}
        for all $t \in [0,T]$, $x,y\in S$ and $\omega \in \Omega$.
  \end{enumerate}
\end{hypo}
For the existence of variational solutions (in the sense of Definition \ref{def:soln_mult} below) we further require
\begin{hypo}\label{hyp:noise_2}
  \begin{enumerate}
   \item[\textup{(B3)}] There is an $h \in L^1([0,T]\times\Omega)$ such that
          \begin{equation*}\|B_t(x)\|_{L_2(U,S)}^2 \le C \|x\|_S^2 + h_t,\quad\textup{($S$-growth)},\end{equation*}
        for all $t \in [0,T]$, $x\in S$ and $\omega \in \Omega$.
  \end{enumerate}
\end{hypo}

\begin{defi}\label{def:soln_mult}
  We say that a continuous $\{\hat{\Fscr}_t\}_{t\ge 0}$-adapted stochastic process $X:[0,T]\times\Omega\to H$ is a solution to
  \begin{equation}\label{maine:mult_noise}\begin{split}
    dX_t + A(t,X_t)\,dt &\ni B_t(X_t)\,dW_t \\
     X(0) &= x 
  \end{split}\end{equation}
  if $X \in L^2(\Omega;C([0,T];H))\cap L^2([0,T]\times \Omega;S)$ and $X$ solves the following integral equation in $S^\ast$
    \[ X_t = X_0 -\int_0^t \eta_s \,ds + \int_0^t B_s(X_s)\, dW_s,\]
  $\P$-a.s.\ for all $t \in [0,T]$, where $\eta \in A(\cdot,X)$, $dt\otimes\P$-a.e.
\end{defi}

\begin{thm}[Multiplicative noise]\label{thm:unique_ex_stoch_mult}
  Let $x \in L^2(\Omega,\hat{\Fscr}_0;S)$. Assume that $A$ satisfies Hypothesis \ref{hyp:A} and $B$ satisfies Hypothesis \ref{hyp:noise} and \ref{hyp:noise_2}. Then there exists a unique solution $X$ to \eqref{maine:mult_noise} in the sense of Definition \ref{def:soln_mult} satisfying
    \[ \E\sup_{t \in [0,T]} \|X_t\|_S^2 < \infty \]
  and $X$ is $\P$-a.s.\ right-continuous in $S$.
\end{thm}
\begin{proof}
See section \ref{proof:unique_ex_stoch_mult}.\end{proof}

Using monotonicity of the drift and Lipschitz continuity of the noise, we can extend the existence result to every initial condition $x \in L^2(\Omega,\hat{\Fscr}_0;H)$ and driving noise taking values in $L_2(U,H)$ in a limiting sense.
\begin{defi}\label{def:stoch_soln_limit}
  An $\{\hat{\Fscr}_t\}_{t \in [0,T]}$-adapted stochastic process $X \in L^2(\Omega;C([0,T];H))$ is a limit solution to \eqref{maine:mult_noise} with starting point $x \in L^2(\Omega,\hat{\Fscr}_0;H)$ if for all approximations $x^\delta \in L^2(\Omega,\hat{\Fscr}_0;S)$ with $x^\delta \to x$ in $L^2(\Omega,\hat{\Fscr}_0;H)$ and $B^\delta$ satisfying Hypothesis \ref{hyp:noise} and \ref{hyp:noise_2} with $B^\delta(u) \to B(u)$ in $L^2([0,T]\times\Omega;L_2(U,H))$ for all $u \in S$ we have 
      $$X^\delta \to X,\quad \text{in } L^2(\Omega;C([0,T];H)).$$
\end{defi}

\begin{thm}[Multiplicative noise for all initial conditions]\label{thm:unique_ex_stoch_mult_all_ic}
   Let $x \in L^2(\Omega,\hat{\Fscr}_0;H)$. We assume that $A$ satisfies Hypothesis \ref{hyp:A} and $B$ satisfies Hypothesis \ref{hyp:noise}. Then there exists a unique limit solution $X$ to \eqref{maine:mult_noise} in the sense of Definition \ref{def:stoch_soln_limit}.
\end{thm}
\begin{proof}
See section \ref{proof:unique_ex_stoch_mult_all_ic}.\end{proof}

\section{Ergodicity}\label{sec:ergodicity}
In the following, we assume $A$, $B$ to be independent of $(t,\omega) \in [0,T] \times \Omega$ and to satisfy Hypotheses \ref{hyp:A} and \ref{hyp:noise_2} with $f$, $\gamma$ and $h$ being constant. We further restrict to the case of additive Wiener noise. Since $B$ satisfies Hypothesis \ref{hyp:noise_2} the process $BW_t$ is a trace class Wiener process in $S$ in the following denoted by $W^B_t$. We consider the canonical realization of its two-sided extension: $\Omega := C(\R;S)$, $\pi_t(\omega) = \omega(t)$ for all $\omega \in \Omega$, $t \in \R$, $\Fscr_t := \sigma\{\pi_s|\ s \in (-\infty,t]\}$, $\Fscr := \sigma\left\{\bigcup_{t\in \R} \Fscr_t\right\}$ and let $\P$ be the law of $W^B$ on $\Omega$. We define the Wiener shift to be $\theta_t (\omega) := \omega(t+\cdot)-\omega(t)$. Then $(\Omega,\{\Fscr_t\}_{t\in \R},\{\theta_t\}_{t\in \R},\P)$ is a metric dynamical system, the evaluation process $\pi_t$ is a trace-class Wiener process, which by abuse of notation is again denoted by $W^B_t$ and $W^B_t$ 
satisfies Hypothesis \ref{hyp:N}. Let $\{\hat\Fscr_t\}$ be the right-continuous
completion of $\{\Fscr_t\}$. We consider evolution inclusions of the form
\begin{equation}\begin{split}\label{eqn:spde_erg}
  dX_t + A(X_t)\,dt &\ni dW^B_t, \quad t > 0\\
  X_0             &= x.
\end{split}\end{equation}
We denote by $\mathscr{B}(H)$ the set of all Borel measurable subsets of $H$, by $\Bscr_b(H)$ (resp.\ $C_b(H)$) the Banach space of all bounded, measurable (resp.\ continuous) functions on $H$ equipped with the supremum norm and by $\Lip_b(H)$ the space of all bounded Lipschitz continuous functions on $H$. By $\mathscr{M}_1$ we denote the set of all Borel probability measures on $H$. For a semigroup $\{P_t\}$ on $\Bscr_b(H)$ we define the dual semigroup $\{P_t^*\}$ on $\mathscr{M}_1$ by $P_t^* \mu(B) := \int_H P_t \mathbbm{1}_B d\mu$, for $B \in \mathscr{B}(H)$. A measure $\mu \in \mathscr{M}_1$ is said to be invariant for the semigroup $\{P_t\}$ if $P_t^* \mu = \mu$, for all $t \ge 0$. For $T > 0$ and $\mu \in \mathscr{M}_1$ we define
  \begin{equation}\label{def:QT} Q^T\mu := \frac{1}{T} \int_0^T P_r^*\mu\,dr \end{equation}
and we write $Q^T(x,\cdot)$ for $\mu = \delta_x$. We endow $ \mathscr{M}_1$ with the topology of weak convergence, i.e.\ $w$-$\lim_{n\to\infty} \mu_n = \mu$ for a sequence $\mu_n \in  \mathscr{M}_1$ iff $\int_H \varphi(x) d\mu_n(x) = \int_H \varphi(x) d\mu(x)$ for all bounded, Lipschitz functions $\varphi$ on $H$. Recall:
\begin{defi}\label{def:weak_erg}
  A semigroup $\{P_t\}$ is called weak-$*$ mean ergodic if there exists a measure $\mu_* \in \mathscr{M}_1$ such that
   \[ \text{$w$-}\lim_{T \to \infty} Q^T\nu = \mu_*, \]
  for all $\nu \in \mathscr{M}_1$.
\end{defi}
Let $X(t,s;\omega)x$ be the solution to \eqref{eqn:spde_erg} starting at time $s \in \R$ in $x \in H$, with respect to the Wiener process $\tilde W^B_t = W^B_{t+s} - W^B_s$ (cf.\ \cite[section 4.3, p.\ 94]{PrRoe}). We choose this notation in order to emphasize that $X(t,s;\omega)$ defines a stochastic flow on $H$ in accordance to \cite[Section 2]{CDF97}.

 We define
  \[ P_x := (X(\cdot,0;\cdot)x)_*\P,\]
to be the law on $C([0,\infty),H)$ of the solution $X(\cdot,0;\cdot)x$ viewed as a random variable taking values in $C([0,\infty),H)$. By It\={o}'s formula
\begin{equation}\label{contr}
  \E \sup_{t \in [0,T]} \|X(t,0;\cdot)x - X(t,0;\cdot)y\|_H^2 \le \|x-y\|_H^2. 
\end{equation}

\begin{prop}\label{prop:markov_gen}
  The family $\{P_x\}_{x \in H}$ defines a time-homogeneous Markov process on $C([0,\infty),H)$ with respect to the filtration $\{\hat\Fscr_t\}_{t \in \R_+}$, i.e.
    \[ \E_x\left[F(\pi_{t+s})\Big|\hat\Fscr_s\right] =  \E_{\pi_s}[F(\pi_{t})],\quad \forall F \in \mathscr{B}_b(H),\ P_x\text{-a.s.}\]
\end{prop}
\begin{proof}
See section \ref{sec:proof:ergo}.
\end{proof}

We define
  $$ P_t F (x):= \E\left[F(X(t,0;\cdot)x)\right],\quad F \in C_b(H),\ x\in H$$
to be the Feller semigroup associated to the stochastic flow $X(t,s;\omega)x$. By Proposition \ref{prop:markov_gen} the semigroup property is satisfied. The so-called $e$-property 
  \begin{equation}\label{eproperty}\abs{P_t F(x)-P_t F(y)}\le\Lip(F)\norm{x-y}_H,\quad x,y\in H,\end{equation}
for all $F \in \Lip_b(H)$ follows from \eqref{contr}.
Consider the assumption
\begin{hypo}\label{hyp:A_2}
  \begin{enumerate}
     \item[\textup{(A5)}] There are constants $C, c > 0$ such that
      \[2 \dualdel{S}{y}{x} \ge c\|y\|_{S^*} - C, \]
      for all $[x,y] \in A$.
  \end{enumerate}
\end{hypo}

Let $u(\cdot)x \in C([0,T];H)$ denote the unique limit solution to
\begin{equation}\label{main_determ}\begin{split}
  du_t + A(u_t) dt &\ni 0, \\
  u_0              &= x \in H. 
\end{split}\end{equation}
The unique existence of such a solution $u$ follows from Theorem \ref{thm:unique_ex_det_smooth} and Theorem \ref{thm:unique_ex_det_all} with $g \equiv 0$.

\begin{hypo}\label{hyp:extinction}
Suppose that there exists a measurable \emph{Lyapunov function} $\Theta:H\to[0,+\infty]$ satisfying
  \[\dualdel{S}{y}{x}\ge\Theta(x),\quad\forall x\in S,\;\;\forall y\in A(x)\]
with bounded sublevel sets $K_R = \{x \in H|\ \Theta(x) \le R \}, R \in \R$ such that 
  \[ \lim_{t\to+\infty}\sup_{x \in K_R} \|u(t)x\|_H = 0, \]
for all $R \ge 0$.
\end{hypo}

Simple sufficient conditions for Hypothesis \ref{hyp:extinction} are given by

\begin{rem}\label{rem:suff_L2}
  \begin{enumerate}
   \item Suppose that there is a constant $c > 0$ such that
     \begin{equation}\label{extinction_sufficient}2 \dualdel{S}{y}{x} \ge c \|x\|_H^\alpha, \end{equation}
  for some $1 \le \alpha < 2$ and all $[x,y] \in A$. Then Hypothesis \ref{hyp:extinction} holds.
   \item Suppose there is a Lyapunov function $\Theta$ with compact sublevel sets and $\lim_{t\to+\infty}\norm{u(t)x}_H=0$ for all $x \in H$. Then Hypothesis \ref{hyp:extinction} holds.
  \end{enumerate}  
\end{rem}
\begin{proof}
See section \ref{sec:proof:ergo}.
\end{proof}

Then, using Theorem \cite[Theorem 2]{KPS}, we prove
\begin{thm}\label{thm:ergodicity}
  Assume Hypotheses \ref{hyp:A}, \ref{hyp:noise_2}, \ref{hyp:A_2}, \ref{hyp:extinction}. Then there exists a unique invariant measure $\mu_*$ for $\{P_t\}$ and
  \begin{enumerate}
   \item The semigroup $\{P_t\}$ is weak-$*$ mean ergodic.
   \item For any $\psi \in \Lip_b(H)$ and $\mu \in \mathscr{M}_1$ the weak law of large number holds, i.e.
         \[ \lim_{T \to \infty} \frac{1}{T} \int_0^T \psi(\pi_s)\, ds = \int_H \psi \,d\mu_*, \]
         in $\P_\mu$-probability, where $\P_\mu$ is the law of the Markov process $\pi$ started with initial distribution $\mu$.
  \end{enumerate}
\end{thm}
\begin{proof}
See section \ref{sec:proof:ergo}.
\end{proof}

\section{Proofs}

\subsection{Deterministic case}

\subsubsection{Deterministic case with initial datum \texorpdfstring{$x \in S$}{in S} (Theorem \ref{thm:unique_ex_det_smooth})}\ \\

We need the following definition, cf. \cite{Zyg}.

\begin{defi}\label{def:superpositionally-measurable}
Let $(T,\Ascr,\mu)$ be a complete $\sigma$-finite measure space, let $X$, $Y$ be Polish spaces.
A multi-valued function $F:T\times X\to 2^Y$ is called \emph{product measurable} if it is $\Ascr\otimes\Bscr(X)$-measurable, where $\Bscr(X)$ is the Borel $\sigma$-algebra of $X$ (see Definition \ref{defi:measurable} for the meaning of \emph{measurability} in this context) and is called \emph{superpositionally measurable} if, for every
$\Ascr$-measurable multi-valued function $G:T\to 2^X$ with non-empty closed values, the
multi-valued function $F_G:T\to 2^Y$ defined by the superposition $F_G(t)=F(t,G(t))$ is
$\Ascr$-measurable, where $F(t,G(t))$ denotes the union of sets $F(t,x)$, where $x\in G(t)$.
\end{defi}

We first prove that the approximating equation \eqref{pereq} has a unique solution:
\begin{prop}\label{prop:unique_ex_approx}
  Assume Hypothesis \ref{hyp:A}: \textup{(A1)}, \textup{(A2)}, \textup{(A4)}, Hypothesis \ref{hyp:g} and let $x \in H$. Then equation \eqref{pereq}, i.e.
 \begin{equation*}\left\{\begin{aligned}
    \frac{d}{dt}Y_t^\eps + \eps i_S(Y_t^\eps) &\in -A(t,Y_t^\eps+g_t),\quad t>0,\\
    Y_0^\eps&=x
\end{aligned}\right.\end{equation*}
has a unique solution $Y^\eps \in C([0,T];H)$ in the sense that 
    \[ Y^\eps_t = x + \int_0^t \zeta^\eps_s \,ds,\quad \forall t \in [0,T],\]
as an equation in $S^*$, where $\zeta^\eps \in L^2([0,T];S^\ast)$ is such that $\zeta^\eps_t \in -\eps i_S(Y_t^\eps)-A(t,Y_t^\eps+g_t)$ for a.e. $t\in [0,T]$.
\end{prop}
\begin{proof}
   We aim to apply Theorem \ref{thm:Hupapa} with
   \[F(t,x) = -A(t,x+g_t)-\frac{\eps}{2}i_S(x),\]
   \[ J(t,x) = \frac{\eps}{2} i_S(x),\] and $u_0 = x$. Obviously (J1)--(J4) are satisfied. We check the conditions on $F$: 
     
  (F1): A consequence of maximal monotonicity (A1), see \cite[Ch.\ 2, Proposition 1.1]{Barb2}.
  
  (F2): Let $x \in S$.
       By \cite[Ch. 3, Theorem 1.28]{HuPapa1} and maximal monotonicity of $x \mapsto A(t,x)$ for all $t \in [0,T]$, $x \mapsto A(t,x)$ is strongly-to-weakly upper-semicontinuous. (A4) and \cite[Theorem 1, Theorem 2]{Zyg} imply that $(t,x)\mapsto A(t,x)$ is product measurable and hence superpositionally measurable, see Definition \ref{def:superpositionally-measurable}. $F+(\eps/2)i_S$ is the composition of the c\`adl\`ag function $t \mapsto (t,x+g_t)$ and $A$, thus $t \mapsto F(t,x)$ is measurable.
      
  (F3): Since $A(t,\cdot)$ is maximal monotone so is $F(t,\cdot)$ for each $t \in [0,T]$, which implies sequential closedness in $S \times S_w^*$, see \cite[Ch. 2, Proposition 1.1]{Barb2}.

  (F4): Obvious by (A2) and Hypothesis \ref{hyp:g}. 

  (F5): We first note that (A1) combined with (A2) implies a weak coercivity property for $A$:
    \begin{align}\label{eqn:weak_H_coerc}
       \dualdel{S}{y}{x} 
       &= \dualdel{S}{y-z}{x-0} + \dualdel{S}{z}{x} \\
       &\ge - \|z\|_{S^*} \|x\|_S \ge - f(t)\|x\|_S, \nonumber
    \end{align}
    for all $t \in [0,T]$ and $[x,y] \in A(t)$ and $z \in A(t,0)$. For $t \in [0,T]$, $[x,y] \in F(t)$, $z \in A(t,x+g_t)$: 
     \begin{align*}
        \dualdel{S}{y}{x} 
        &= - \dualdel{S}{z}{x} - \frac{\eps}{2} \|x\|_S^2 = - \dualdel{S}{z}{x+g_t} - \dualdel{S}{z}{g_t} - \frac{\eps}{2} \|x\|_S^2 \\
        &\le f(t) \|x+g_t\|_S + \|z\|_{S^*}\|g_t\|_S - \frac{\eps}{2} \|x\|_S^2 \\
        &\le C(\eps) ( f(t)^2+ f(t) \|g_t\|_S + \|g_t\|_S^2 )
     \end{align*}
     and the last term is in $L^1([0,T])_{+}$ by (A2) and Hypothesis \ref{hyp:g}.
 
  Application of Theorem \ref{thm:Hupapa} thus yields the existence of a solution and monotonicity of $F, J$ implies uniqueness.
\end{proof}

\begin{lem}\label{boundlem}
  Under Hypothesis \ref{hyp:A} and \ref{hyp:g} for starting points $x \in S$, we have that
    \[\norm{Y_t^\eps}_S^2 \le C\norm{x}_S^2  + C \int_0^t e^{-Cs} \left(\gamma(s) + f_s^2 + \|T^{3/2}g_s\|_{H}^2 \right) ds < \infty.\]
\end{lem}
\begin{proof}
  The proof is based on the smooth approximation of $\|\cdot\|_S$ by the norms $\|\cdot\|_n$. As outlined in Section \ref{sec:det_case} we first apply It\=o's formula for $\|\cdot\|_n^2$ and then take the limit $n \to \infty$.

  Based on the definition of the spaces $H_n$ in Section \ref{sec:det_case}, we get a sequence of new Gelfand triples
  \[S\subset H_n \subset S^\ast.\]
Moreover,
  \[T_n\rs{S}:S\to S\]
is continuous and the Riesz-map $i_S: S \to S^*$ is given by the extension of $T$ to $T: S \to S^*$. Thus
  $$\label{tn2} 2\dualdel{S}{i_S(x)}{T_n(x)} = 2(x,T_nx)_S = 2(T^{1/2}x,T_nT^{1/2}x)_H \ge 0, \quad\forall x\in S,$$
which will be used below.

  By Proposition \ref{prop:unique_ex_approx} we know that $Y^{\eps} \in L^2([0,T];S)$. We apply the chain-rule \cite[\S III.4]{Show} to $\norm{\cdot}^2_n$ to obtain that
  \begin{align*}
    \norm{Y^{\eps}_t}_n^2 e^{-Kt}
    &\le \norm{x}_n^2 - 2 \int_0^t e^{-Ks}\dualdel{S}{\eta_s^\eps+\eps i_S(Y^{\eps}_s)}{T_n(Y^{\eps}_s)}\,ds - K \int_0^t e^{-Ks} \norm{Y^{\eps}_s}_n^2\,ds \nonumber\\
    &=   \norm{x}_n^2 - 2 \int_0^t e^{-Ks} \dualdel{S}{\eta_s^\eps}{T_n(Y^{\eps}_s+g_s)}\,ds + 2 \int_0^t e^{-Ks} \dualdel{S}{\eta_s^\eps}{T_n(g_s)}\,ds \nonumber\\
      &\hskip15pt - 2\eps \int_0^t e^{-Ks} (Y^{\eps}_s,T_n(Y^{\eps}_s))_S\,ds - K \int_0^t e^{-Ks} \norm{Y^{\eps}_s}_n^2\,ds\nonumber\\
    &\le   \norm{x}_n^2 + \int_0^t e^{-Ks} \left( \gamma(s)+ C \|Y^{\eps}_s+g_s\|_S^2 \right) \,ds + \int_0^t e^{-Ks} \|\eta_s^\eps\|_{S^*}^2 \,ds \\
      &\hskip15pt + \int_0^t e^{-Ks} \|T_n g_s\|_S^2 \,ds - K \int_0^t e^{-Ks} \norm{Y^{\eps}_s}_n^2 ds\nonumber \\
    &\le   \norm{x}_n^2 + 2C \int_0^t e^{-Ks} \|Y^{\eps}_s\|_S^2 \,ds - K \int_0^t e^{-Ks} \norm{Y^{\eps}_s}_n^2 \,ds \nonumber\\
      &\hskip15pt + C \int_0^t e^{-Ks} \left(\gamma(s) + f_s^2 + \| T g_s\|_{S}^2 \right) ds, \nonumber  \end{align*}
  where $\eta_s^\eps \in A(Y^\eps_s + g_s)$ $ds$-a.e. Taking $K=2C$ and $n \to \infty$ yields
  \begin{align}\label{eqn:S_bound}
    \norm{Y^{\eps}_t}_S^2 e^{-Ct} \le \norm{x}_S^2  + C \int_0^t e^{-Cs} \left(\gamma(s) + f_s^2 + \|T^{3/2}g_s\|_{H}^2 \right) ds.
  \end{align}
\end{proof}

We are now ready to prove unique existence of solutions for initial data $x \in S$, i.e.
\begin{proof}[Proof of Theorem \ref{thm:unique_ex_det_smooth}]\label{proof:unique_ex_det_smooth}
  For $\eps > 0$ let $Y^\eps$ denote the solution to \eqref{pereq} with corresponding selection $\eta^\eps \in A(\cdot,Y^\eps+g)$. By Lemma \ref{boundlem} $\norm{Y^\eps}_{L^\infty([0,T];S)}\le C.$ Hence, there is a sequence $\{\eps_n\}$, $\eps_n\to 0$ such that 
    $$Y^{\eps_n}\rightharpoonup^\ast Y$$
  weakly-$\ast$ in $L^\infty([0,T];S)$. By the chain-rule (cf.\ \cite[\S III.4]{Show})
  \begin{align}\label{eqn:eps->0}
    \norm{Y_t^\eps-Y_t^\delta}^2_H 
    &\le - 2\int_0^t\dualdel{S}{\eta_s^\eps-\eta_s^\delta}{Y_s^\eps-Y_s^\delta}\,ds \\
      &\hskip12pt - 2\int_0^t\dualdel{S}{\eps i_S(Y_s^\eps)-\delta i_S(Y_s^\delta)}{Y_s^\eps-Y_s^\delta}\,ds,\quad t\in [0,T], \nonumber
  \end{align}
  for all $0<\delta<\eps$. Hence, 
    \[\int_0^T\dualdel{S}{\eps i_S(Y_s^\eps)-\delta i_S(Y_s^\delta)}{Y_s^\eps-Y_s^\delta}\,ds \le 0.\]
  Since $\norm{Y^\eps}_{L^\infty([0,T];S)}\le C$  we may apply \cite[Lemma 1.4, (ii)]{BCP} to obtain that $Y^{\eps_n}\rightharpoonup Y$ weakly in $L^2([0,T];S)$ and
    \begin{equation}\label{strongYeq}\int_0^T\norm{Y^{\eps_n}_t}_S^2\,dt\to\int_0^T\norm{Y_t}_S^2\,dt.\end{equation}
  Hence $Y^{\eps_n}\to Y$ strongly in $L^2([0,T];S)$. Then \eqref{eqn:eps->0} implies $Y^{\eps_n} \to Y$ in $C([0,T];H)$. By (A2) and Lemma \ref{boundlem} we can choose a further subsequence (again denoted by $\eps_n$) such that
    \[ \eta^{\eps_n} \rightharpoonup \eta, \quad \text{in } L^2([0,T];S^*). \]
  In order to identify the limit $\eta$, we consider the operator $\tilde A: L^2([0,T];S) \to 2^{L^2([0,T];S^*)}$ defined by 
   $$\tilde A: x \mapsto \{ y \in L^2([0,T];S^*)|\ y(t) \in A(t,x(t)) \text{ for a.e. } t \in [0,T] \}. $$
  By Proposition \ref{prop:graphclosed} $\tilde A$ is strongly-weakly closed, i.e.\ $\tilde A$ is sequentially closed in $L^2([0,T];S) \times L^2_w([0,T];S^*)$. Thus $[Y+g,\eta] \in \tilde A$ and we deduce $[Y_t+g_t,\eta_t] \in A(t)$ for a.e.\ $t \in [0,T]$. By monotonicity the solution is unique and thus the whole sequence (net) $\{Y^\eps\}$ converges strongly to $Y$ in $L^2([0,T];S)$. Since $Y$ is a unique solution to \eqref{eqn:trans_det} $X_t := Y_t + g_t$ is a unique solution to \eqref{eqn:det}.

  We now prove $X \in L^\infty([0,T];S)$ and $X$ is right-continuous in $S$. By the same calculation as for \eqref{eqn:S_bound} we obtain
  \begin{equation}\label{eqn:S_bound_2}
     \norm{Y_t}_S^2 e^{-Ct} \le \norm{Y_s}_S^2 e^{-Cs}  + C \int_s^t e^{-Cr} \left(\gamma(r) + f_r^2 + \|T^{3/2} g_r\|_{H}^2 \right)\, dr. 
  \end{equation}
  In particular,
    $$ \sup_{t \in [0,T]} \|Y_t\|_S^2 < \infty$$
  which due to compactness of $S \subseteq H$ and continuity of $t \mapsto Y_t$ in $H$ implies weak continuity of $Y$. Let $t_n \in [0,T]$ with $t_n \downarrow t$. By \eqref{eqn:S_bound_2}
    $$ \norm{Y_{t_n}}_S^2 e^{-C{t_n}} \le \norm{Y_t}_S^2 e^{-Ct}  + C \int_t^{t_n} e^{-Cs} \left(\gamma(s) + f_s^2 + \|T^{3/2} g_s\|_{H}^2 \right) ds.  $$
  Hence, $\limsup_{n \to \infty} \norm{Y_{t_n}}_S^2 \le \norm{Y_t}_S^2$. By weak continuity we conclude $Y_{t_n} \to Y_t$ in $S$.
   
\end{proof}

\vspace{4mm}
\subsubsection{Deterministic case with initial datum \texorpdfstring{$x \in H$}{in H} (Theorem \ref{thm:unique_ex_det_all})}\ \\

We prove Theorem \ref{thm:unique_ex_det_all}. Let $x^\delta$, $\bar x^\eps \in S$ be two approximations of $x$ in $H$ and $Y^\delta$, $\bar Y^\eps$ be the variational solutions constructed in Theorem \ref{thm:unique_ex_det_smooth} corresponding to $x^\delta, \bar x^\eps \in S$. By the chain-rule
  \begin{align*}
    \|Y_t^\delta-\bar Y_t^\eps\|_H^2 
    &= \|x^\delta-\bar x^\eps\|_H^2 - 2\int_0^t \dualdel{S}{\eta^\delta_r-\bar\eta^\eps_r}{Y_r^\delta-\bar Y^\eps_r}\,dr \\
    &\le \|x^\delta-\bar x^\eps\|_H^2,
  \end{align*}
  where $\eta^\delta_r \in A(Y^\delta_r+g_r)$, $\bar\eta^\eps_r \in A(\bar Y^\eps_r+g_r)$. Thus $\{Y^\delta\}$ is a Cauchy sequence in $C([0,T];H)$ and every sequence of approximative solutions $\{\bar Y^\eps\}$ converges to the same limit $Y$ in $C([0,T];H)$.

\vspace{4mm} 
\subsection{Stochastic evolution inclusions with additive noise}\label{sec:pathwise_additive_noise}\ \\

Let $x \in L^0(\Omega,\Fscr_0;S)$ and $X$ be the corresponding pathwise solution to \eqref{aeq}. It remains to prove $\{\Fscr_t\}_{t \in [0,T]}$-adaptedness of $X$. 

Under the assumptions Theorem \ref{thm:unique_ex_det_smooth} we have shown that there is a unique solution $Y$ to \eqref{eqn:trans_det}. We now prove that the solution map 
\begin{align*}
  F: S \times \left(L^2([0,T];S)\cap L^2_w([0,T];D(T^{3/2}))\right)  &\to C([0,T];H) \\
      (x,g)                     &\mapsto Y,
\end{align*}
is continuous (here, the subscript ``$\phantom{\cdot}_w$'' denotes the weak Hilbert topology). Let $x^n \to x$ in $S$ and $g^n \to g$ in $L^2([0,T];S)$ with $g^n\rightharpoonup g$ weakly in $L^2([0,T];D(T^{3/2}))$. The bound in Lemma \ref{boundlem} for $Y^n$, $Y$ does not depend on $n$ since $g^n$ is uniformly bounded in $L^2([0,T];D(T^{3/2}))$. Hence $\eta_t^n \in A(t,Y_t^n+g_t^n)$, $\eta_t \in A(t,Y_t+ g_t)$ are uniformly bounded in $L^2([0,T];S^*)$. By the chain rule (\cite[\S III.4]{Show})
  \begin{align*}
    \norm{Y_t^n-Y_t}^2_H 
    &= - 2 \int_0^t\dualdel{S}{\eta_s^n-\eta_s}{Y_s^n-Y_s}\,ds \\
    &\le - 2 \int_0^t\dualdel{S}{\eta_s^n-\eta_s}{g_s- g_s^n}\,ds \\
    &\le 2 \|\eta_s^n-\eta_s\|_{L^2([0,T];S^*)}  \|g-g^n\|_{L^2([0,T];S)} \to 0,
  \end{align*}
which proves the claimed continuity.

Let $N$ be as in Section \ref{sec:add_noise}. By Kuratowski's Theorem \cite[Ch.\ I.3]{parthasarathy1967} the restriction $N\rs{[0,t]}: \Omega \to L^2([0,t];D(T^{3/2}))$ is $\Fscr_t$-measurable for all $t \in [0,T]$. Hence, continuity of the solution map $F$ implies $\{\Fscr_t\}_{t \in [0,T]}$-adaptedness of $Y$ and thus of $X$.

\vspace{4mm} 
\subsection{Stochastic evolution inclusions with multiplicative noise}\ \\
We will now prove Theorem \ref{thm:unique_ex_stoch_mult} and Theorem \ref{thm:unique_ex_stoch_mult_all_ic}. The proof consists of two main steps. First we consider the case of additive noise satisfying the weaker regularity properties in Hypothesis \ref{hyp:noise}, \ref{hyp:noise_2} as compared to Hypothesis \ref{hyp:g}. In the second step we use the unique existence of solutions for additive noise in order to construct solutions to the case of multiplicative noise, using a fixed point argument.

\vskip20pt 
\subsubsection{Stochastic evolution inclusions with additive Wiener noise}\ \\

For the rest of this Section assume $B$ to be independent of $x \in S$.
\begin{prop}\label{prop:unique_ex_stoch_add}
  Let $x \in L^2(\Omega,\hat{\Fscr}_0;S)$ and assume Hypotheses \ref{hyp:A}, \ref{hyp:noise} and \ref{hyp:noise_2}. Then there is a solution $X$ to \eqref{maine:mult_noise} satisfying 
    \[ \E\sup_{t \in [0,T]}\|X_t\|_S^2 <\infty\]
  and $X$ is $\P$-a.s.\ right-continuous in $S$. If $x \in L^2(\Omega,\hat\Fscr_0;H)$ and Hypotheses \ref{hyp:A} and \ref{hyp:noise} are satisfied, then there is a unique limit solution to \eqref{maine:mult_noise}.  
\end{prop}
\begin{proof}
  We first make some general remarks concerning the approximation of elements $x \in S$. Since $T: D(T) \subseteq H \to H$ is a linear, positive definite, self-adjoint operator with compact resolvent there is an orthonormal basis of eigenvectors $e_i \in S$ of $H$. Let $\Hscr_m = $ span$\{e_1,...,e_m\}$ and $P_m: H \to \Hscr_m$ be the orthogonal projection onto $\Hscr_m$ in $H$. Since $(\cdot,\cdot)_S = (T^{1/2}\cdot,T^{1/2}\cdot)_H$ the restriction of $P_m$ to $S$ is the orthogonal projection onto $\Hscr_m$ in $S$. Moreover, $P_m$ can be extended continuously to $P_m: S^* \to S$ with $\|P_m y\|_{S} \le C_m\|y\|_{S^*}$.

  We consider the approximating equations
  \begin{equation}\label{eqn:stoch_add_noise_approx}\left\{\begin{aligned}
    dX^m_t+ A(t,X^m_t)\,dt &\ni B^m_t\, dW_t \\
    X_0 &= x,
  \end{aligned}\right.\end{equation}
  where $B^m := P_m B$. Then
    \[ \|P_m B\|_{L^2([0,T]\times\Omega;L_2(U,S))}^2 \le \|B\|_{L^2([0,T]\times\Omega;L_2(U,S))}^2 \]
  and $\|T P_m B\|_{L^2([0,T]\times\Omega;L_2(U,S))}^2 \le C_m \|B\|_{L^2([0,T]\times\Omega;L_2(U,S))}^2$.
  Hence 
    \[N_t^m(\omega) := \int_0^t B^m_s \,dW_s(\omega)\]
  satisfies Hypothesis \ref{hyp:g} for a.a.\ $\omega \in \Omega$. By dominated convergence $B^m \to B$ in $L^2([0,T]\times\Omega;L_2(U,H))$ and by Theorem \ref{thm:unique_ex_det_smooth} there is a solution $X^m(\omega)$ to \eqref{eqn:stoch_add_noise_approx} for a.a.\ $\omega \in \Omega$. 

  \begin{lem}\label{lemma:apriori_S_bound}
    There is $C > 0$ such that
      \[ \E\left[ \sup_{t \in [0,T]} \|X^m_t\|_S^2 \right] \le 4e^{CT} \left( \E\|x\|_S^2 + \E \int_0^T (h_s+\gamma_s)\, ds  \right),\]
    where $\gamma$ is as in Hypothesis \ref{hyp:A} and $h$ as in Hypothesis \ref{hyp:noise_2}.
    \end{lem}
  \begin{proof}
   We first note that the process $t \mapsto \int_0^t \eta^m_s \,ds$ (where $\eta^m\in A(\cdot,X^m)$ $dt\otimes\P$-a.s.) is progressively measurable and that this is sufficient for the application of It\={o}'s formula (cf.\ \cite[Theorem 4.2.5]{PrRoe}) for the approximating norm $\norm{\cdot}_n^2$. In fact, the same proof as for \cite[Theorem 4.2.5]{PrRoe} can be used. We obtain by (A3):
   \begin{align*}
      \|X^m_t\|_n^2 \,e^{-Ct}
      &= \|x\|_n^2 - \int_0^t 2 e^{-Cs}\dualdel{S}{\eta_s^m}{T_n(X^m_s)}\, ds + \int_0^t 2 e^{-Cs}(X^m_s,B^m_s\,dW_s)_n  \\
        &\hskip12pt + \int_0^t e^{-Cs}\|B^m_s\|_{L_2(U,H_n)}^2\, ds - C\int_0^t e^{-Cs} \|X^m_s\|_n^2\, ds\\
      &\le \|x\|_n^2+\int_0^t e^{-Cs} \gamma_s\,ds+C\int_0^t e^{-Cs}\norm{X_s^m}_S^2\,ds\\
      &\hskip12pt + \int_0^t 2 e^{-Cs}(X^m_s,B^m_s\,dW_s)_n 
        + \int_0^t e^{-Cs}\|B^m_s\|_{L_2(U,H_n)}^2\, ds \\
        &\hskip12pt - C\int_0^t e^{-Cs} \|X^m_s\|_n^2\, ds.
   \end{align*}
   By the Burkholder-Davis-Gundy inequality, choosing $K$ large enough and taking $n \to \infty$, we obtain
   \begin{align*}
       \E \left[\sup_{t \in [0,T]}\|X^m_t\|_S^2 \, e^{-Ct} \right] \le 2 \left(\E \|x\|_S^2 + \E\int_0^T  e^{-Cs}(2 h_s+\gamma_s) \, ds \right).
   \end{align*}
 \end{proof}
  We proceed with the proof of Proposition \ref{prop:unique_ex_stoch_add}. By It\={o}'s formula and the Burkholder-Davis-Gundy inequality
    \[ \E\left[ \sup_{t \in [0,T]} \|X_t^n-X_t^m\|_H^2\right] \le C\E \int_0^T \|B^n_s - B^m_s\|_{L_2(U,H)}^2 \,ds \to 0,\]
  for $n , m \to \infty$. Hence, $X^n \to X$ in $L^2(\Omega;C([0,T];H))$. By lower semi-continuity of the norm $\|\cdot\|_{L^2(\Omega;C([0,T];S))}$ on $L^2(\Omega;C([0,T];H))$, 
  \begin{equation}\label{eqn:S_bound_add_case}
    \E \left[\sup_{t \in [0,T]} \|X_t\|_S^2\right] \le C \left( \E\|x\|_S^2 + \E \int_0^T \left( h_s+\gamma_s \right)\, ds \right).
  \end{equation}
  By Lemma \ref{lemma:apriori_S_bound} and (A2), for some subsequence $\eta^m \rightharpoonup \bar\eta$ in $L^2([0,T]\times\Omega;S^*)$ and we obtain $\P$-a.s.
    \[ X_t = x - \int_0^t \bar\eta_s ds + \int_0^t B_s dW_s , \quad \forall t \in [0,T]. \]
  As in the proof of Lemma \ref{lemma:apriori_S_bound} we may apply It\={o}'s formula to this equation and to \eqref{eqn:stoch_add_noise_approx}. Subtracting the resulting equations yields
  \begin{align*}
    \E\left[ \|X_t^m\|_H^2-\|X_t\|_H^2 \right]
    = &- \E \left[\int_0^t 2\dualdel{S}{\eta_r^m}{X_r^m} - 2\dualdel{S}{\bar\eta_r}{X_r} dr\right] \\
      &+ \E\int_0^t \|B^m_r - B_r\|_{L_2(U;H)}^2
  \end{align*}
  and taking $\limsup\limits_{m \to \infty}$ we obtain
  \[ \limsup_{m \to \infty} \E \int_0^t \dualdel{S}{\eta_r^m}{X_r^m} \le \E \int_0^t \dualdel{S}{\bar\eta_r}{X_r} dr.\]
  By Proposition \ref{prop:max_monotone}, $A: L^2([0,T] \times \Omega;S) \to L^2([0,T] \times \Omega;S^*)$ is maximal monotone. Thus, Minty's trick implies $[X,\bar{\eta}] \in A$ $dt\otimes\P$-a.e. 

  Right continuity of $X$ in $S$ is shown as in the proof of Theorem \ref{thm:unique_ex_det_smooth}, i.e.\ \eqref{eqn:S_bound_add_case} yields weak continuity and repeating the calculations from Lemma \ref{lemma:apriori_S_bound} for all initial times $s \le t$ we can then deduce right-continuity.

  For general initial conditions $x \in L^2(\Omega,\hat{\Fscr}_0;H)$ and noise satisfying Hypothesis \ref{hyp:noise} only we consider approximations $x^m \in L^2(\Omega,\hat{\Fscr}_0;S)$ and $B^m := P_m B$ with corresponding variational solutions $X^m$. Applying It\={o}'s formula for the difference $\|X^m-X^n\|_H^2$ and using Burkholder's inequality yields
    \[ \E \left[\sup_{t\in [0,T]} \|X^m_t-X^n_t\|_H^2\right] \le C \left( \E\|x^m-x^n\|_H^2 + \E \int_0^T \|B^n_s-B^m_s\|_{L_2(U,H)}^2\, ds \right),\]
  which implies the existence of a limit $X$ of $X^m$ in $L^2(\Omega;C([0,T];H))$. For two approximating solutions $X^m$, $\bar X^m$ by the same argument we obtain
    \[ \E \left[\sup_{t\in [0,T]} \|X^m_t-\bar X^m_t\|_H^2\right] \le C \left( \E\|x^m-\bar x^m\|_H^2 + \E \int_0^T \|B^m_s-\bar B^m_s\|_{L_2(U,H)}^2\, ds \right),\]
  This implies that the limit $X$ does not depend on the approximating sequence.
\end{proof}

\vspace{4mm}
\subsubsection{Proof of Theorem \ref{thm:unique_ex_stoch_mult}}\ \\
\label{proof:unique_ex_stoch_mult}

First let $x \in L^2(\Omega,\hat\Fscr_0;S)$ and $B$ satisfy Hypothesis \ref{hyp:noise} and Hypothesis \ref{hyp:noise_2}. We construct a solution by freezing the noise. For $K > 0$ let 
  \[\Wscr^{K} := \left\{Z \in L^2(\Omega;C([0,T];H))|\ \|Z\|_{L^2(\Omega;C([0,T];S))}^2 \le K\right\}.\]
Since $\|\cdot\|_{L^2(\Omega;C([0,T];S))}^2$ is a lower semi-continuous function on $L^2(\Omega;C([0,T];H))$ the subsets $\Wscr^{K}$ are closed in $L^2(\Omega;C([0,T];H))$. For $Z \in L^2(\Omega;C([0,T];S))$ we have         
  \[ \|B_t(Z_t)\|_{L_2(U,S)}^2 \le C \|Z_t\|_S^2 + h_t \in L^1([0,T]\times\Omega). \]
By Proposition \ref{prop:unique_ex_stoch_add} there exists a unique corresponding solution $X = F(Z) \in L^2(\Omega;C([0,T];S))$ driven by the diffusion coefficients $B(Z)$. We will prove that the solution map 
  \[ F: L^2(\Omega;C([0,T];H)) \to L^2(\Omega;C([0,T];H)) \]
is a contraction for $T > 0$ small enough. Let $Z^{(1)}, Z^{(2)} \in L^2(\Omega;C([0,T];H))$. Then by It\={o}'s formula and the Burkholder-Davis-Gundy inequality
\begin{align*}
  \|F(Z^{(1)})-F(Z^{(2)})\|_{L^2(\Omega;C([0,T];H))}^2 
  &\le \E \int_0^T \|B_s(Z_s^{(1)})-B_s(Z_s^{(2)})\|_{L_2(U,H)}^2\, ds \\ 
  &\le C \E \int_0^T \|Z_s^{(1)}-Z_s^{(2)}\|_H^2\, ds \\
  &\le CT \|Z_s^{(1)}-Z_s^{(2)}\|_{L^2(\Omega;C([0,T];H))}^2.
\end{align*}
Thus, for $T > 0$ small enough $F$ is a contractive mapping. For $Z \in\Wscr^{K}$, by \eqref{eqn:S_bound_add_case}
\begin{align*}
  \E\left[ \sup_{t \in [0,T]} \|X_t\|_S^2\right] 
  &\le 4e^{CT} \left( \E\|x\|_S^2 + \E \int_0^T C\|Z_s\|_S^2 ds + \E \int_0^T \left( h_s + \gamma_s \right) \,ds \right)\\
  &\le 4e^{CT} \left( \E\|x\|_S^2 + C T \|Z_s\|_{L^2(\Omega;C([0,T];S))}^2  +  \E \int_0^T \left( h_s + \gamma_s \right) \,ds \right)\\
  &\le 4e^{CT} \left(\E\|x\|_S^2 + C T K  +  \E \int_0^T \left( h_s + \gamma_s \right)\, ds\right) 
\end{align*}
Choosing $K \ge 8e^{CT} \left( \E\|x\|_S^2 +  \E \int_0^T \left( h_s + \gamma_s \right)\, ds\right)$ thus yields
\begin{align*}
  \E\left[ \sup_{t \in [0,T]} \|X_t\|_S^2\right] 
  &\le \frac{K}{2} + 4Ce^{CT} T K.
\end{align*}
Hence, for $T\ge 0$ small enough, $F$ leaves $\Wscr^{K}$ invariant and is contractive. By Banach's fixed point theorem, there is a unique fixed point $X \in  L^2(\Omega;C([0,T];H))$, i.e.\ $F(X) = X$ or in other words
  \[ dX_t + A(X_t)\,dt \ni B(X_t)\,dW_t. \]
By Theorem \ref{prop:unique_ex_stoch_add} we have
  \[ \E\sup_{t \in [0,T]} \|X_t\|_S^2 < \infty \]
and $X$ is $\P$-a.s.\ right-continuous in $S$.

\vspace{4mm}
\subsubsection{Proof of Theorem \ref{thm:unique_ex_stoch_mult_all_ic}}\ \\
\label{proof:unique_ex_stoch_mult_all_ic}

Let $x \in L^2(\Omega,\hat\Fscr_0;H)$, $B$ satisfying Hypothesis \ref{hyp:noise} only, $x^n \in L^2(\Omega,\hat\Fscr_0;S)$ with $x^n \to x$ in $L^2(\Omega,\hat\Fscr_0;H)$ and $B^m := P_m B$, thus satisfying Hypothesis \ref{hyp:noise} and Hypothesis \ref{hyp:noise_2}. By It\={o}'s formula and Burkholder's inequality 
\begin{align*}
  &\E\sup_{t \in [0,T]} \|X_t^n - X_t^m\|^2_H \, e^{-Ct} \\
  &\le C\left( \E\|x^n-x^m\|_H^2 + \E\int_0^T \|B_s^n(X_s^n)-B_s^m(X_s^m)\|_{L_2(U;H)}^2 ds\right) \\
  &\le C \E\|x^n-x^m\|_H^2 + C\E\int_0^T \|B_s^n(X_s^n)-B_s(X_s^n)\|_{L_2(U;H)}^2  ds  \\
    &\hskip12pt + C\E\int_0^T \|B_s(X_s^m)-B_s^m(X_s^m)\|_{L_2(U;H)}^2  ds \to 0,
\end{align*}
by dominated convergence. Hence, there is a limit $X^n \to X \in L^2(\Omega;C([0,T];H))$. Similar arguments yield the independence of $X$ from the approximating sequence.

\vspace{4mm}
\subsection{Markov processes and ergodicity}\ \\
\label{sec:proof:ergo}
In the following assume that the assumptions considered in Section \ref{sec:ergodicity} are satisfied.

\begin{proof}[Proof of Proposition \ref{prop:markov_gen}:]
  As in \cite[Proposition 4.3.5]{PrRoe} we note that by \eqref{contr} it is enough to show
    \[ \E_x[G(\pi_{t_1},...,\pi_{t_n})F(\pi_{t+s})] = \int_\Omega G(\pi_{t_1}(\omega),...,\pi_{t_n}(\omega)) \E_{\pi_s(\omega)}[F(\pi_t)]\ dP_x(\omega), \]
  for all $0 \le t_1 \le ... \le t_n \le s$, $G: H^n \to \R$ continuous, bounded and $F \in C_b(H)$.
  Equivalently
  \begin{equation}\label{eqn:Markov_prop_for_X}\begin{split}
     &\E[G(X(t_1,0;\cdot)x,...,X(t_n,0;\cdot)x)F(X(t+s,0;\cdot)x)] \\
     &= \int_\Omega G(X(t_1,0;\omega)x,...,X(t_n,0;\omega)x) \E[F(X(t,0;\cdot)X(s,0;\omega)x)]\ d\P(\omega).
  \end{split}\end{equation}   

  Let us first consider the case of regular initial conditions $x \in S$ and additive, pathwise $D(T^{3/2})$-regular noise (Theorem \ref{thm:unique_ex_stoch_add_pathw}). By Corollary \ref{cor:gen_RDS}, $X(t,s;\omega)x$ is a stochastic flow, i.e.
    \[ X(t,s;\omega)x = X(t,r;\omega)X(r,s;\omega)x, \quad \forall s \le r \le t \]
  and a cocycle
    \[ X(t,s;\omega)x = X(t-s,0;\theta_s\omega), \quad \forall s \le t.\]
  Thus:
    \[ X(t,0;\cdot)X(s,0;\omega)x = X(t+s,s;\theta_{-s}\cdot)X(s,0;\omega)x. \]
      Since $X(t+s,s;\cdot)$ is independent of $\hat\Fscr_s$ we conclude
  \begin{align*}
    \E[F(X(t,0;\cdot)X(s,0;\omega)x)] 
    &= \E[F(X(t+s,s;\cdot)X(s,0;\omega)x)] \\
    &= \E[F(X(t+s,s;\cdot)X(s,0;\cdot)x)|\hat\Fscr_s](\omega) \\
    &= \E[F(X(t+s,0;\cdot)x)|\hat\Fscr_s](\omega)
  \end{align*}
  and thus
  \begin{align*}
     &\int_\Omega G(X(t_1,0;\omega)x,...,X(t_n,0;\omega)x) \E[F(X(t,0;\cdot)X(s,0;\omega)x)]\ d\P(\omega) \\
     &= \int_\Omega G(X(t_1,0;\omega)x,...,X(t_n,0;\omega)x) \E[F(X(t+s,0;\cdot)x)|\hat\Fscr_s](\omega) \ d\P(\omega) \\
     &= \E[G(X(t_1,0;\cdot)x,...,X(t_n,0;\cdot)x)F(X(t+s,0;\cdot)x)].
  \end{align*}
    
  For initial conditions $x \in H$ and noise $B$ satisfying Hypothesis \ref{hyp:noise} only, solutions were constructed as limits of pathwise solutions $X^m(\cdot,0,\cdot)x \to X(\cdot,0;\cdot)x$ in $L^2(\Omega;C([0,T];H))$. Using uniform Lipschitz continuity of the pathwise solutions in the initial condition we realize that \eqref{eqn:Markov_prop_for_X} is preserved in the limit.
\end{proof}

Recall that $u(\cdot)x \in C([0,T];H)$ denotes the unique solution to \eqref{main_determ}.

\begin{proof}[Proof of Remark \ref{rem:suff_L2}:]
  (1): The Lyapunov function $\Theta(x) := c\|x\|_H^\alpha$ is measurable and has bounded sublevel sets. It remains to prove convergence to $0$ uniformly on sublevel sets. If $\a =2$ this follows immediately from Gronwall's inequality. Thus, consider $1 \le \a < 2$.

  First let $x \in S \cap B$.  By the chain rule of calculus
   \[ \frac{d}{dt}\|u_t\|_H^2 = 2 \dualdel{S}{-\eta_t}{u_t} \le - c\|u_t\|_H^\alpha = - c\left( \|u_t\|^2_H \right)^\frac{\alpha}{2},\quad \text{ for a.e.\ } t \in [0,\infty),\]
  where $\alpha$ is as in \eqref{extinction_sufficient} and $\eta_t \in A(u_t)$ for a.e.\ $t \in [0,T]$. Hence, informally $f(t) := \|u_t\|_H^2$ is a subsolution to the ordinary differential equation
   \[ f'(t) = -cf(t)^\frac{\alpha}{2},\quad \text{for a.e. } t \in [0,T].\]
  Hence
   \[ \|u_t\|_H^2 \le \left( \left(\|x\|_H^{2-\alpha} - ct \frac{2-\a}{2} \right) \vee 0 \right)^\frac{2}{2-\alpha} \le \left( \left(\|B\|_H^{2-\alpha} - ct \frac{2-\a}{2}  \right) \vee 0 \right)^\frac{2}{2-\alpha}. \]
  By continuity in the initial condition the same inequality holds for all $x \in B$. We conclude $u_t \equiv 0$ for all $t \ge T_B := \|B\|_H^{2-\alpha}\frac{2}{c(2-\a)}$.

  (2): By monotonicity of $A$, $u_t$ is non-expansive, i.e.
      $$\|u(t)x-u(t)y\|_H \le \|x-y\|_H, \quad \forall x,y \in H.$$
      Therefore, the convergence to $0$ is uniform on compact sets.  
\end{proof}     

In order to prove Theorem \ref{thm:ergodicity}, we need some preparation. We start by proving stochastic stability for equation \eqref{eqn:spde_erg}.
\begin{lem}\label{stoch_stable_lemma}
Suppose that Hypotheses \ref{hyp:A}, \ref{hyp:noise_2}, \ref{hyp:A_2} hold.
  For each $T > 0$, $\eps >0$ and $B \subseteq H$ bounded we have
   \[ \P[\|X(T,0;\cdot)x-u(T)x\|_H^2 \le \eps] > 0,\]
  uniformly for all $x \in B$.
\end{lem}
\begin{proof}
  Since $W^B_t$ is a trace class Wiener process on $S$, for each $\delta > 0$, $T > 0$ we can find a subset $\Omega^\delta \subseteq \Omega$ of positive mass such that $\sup\{\|W^B_t(\omega)\|_S|\ t \in [0,T]\} < \frac{\delta}{2}$ for all $\omega \in \Omega^\delta$. Let $x \in S$. For $\omega \in \Omega^\delta$ we have by (A5)
  \begin{align*}
     &\|Y(t,0;\omega)x\|_H^2 
     = \|x\|_H^2 - 2 \int_0^t \dualdel{S}{\eta_r}{X(r,0;\omega)x-W^B_r(\omega)} dr \\
     &\le \|x\|_H^2 - c \int_0^t \|\eta_r\|_{S^*} dr - 2 \int_0^t \dualdel{S}{\eta_r}{-W^B_r(\omega)} dr + C t\\
     &\le \|x\|_H^2 - \left( c-2\sup_{r\in[0,T]}\|W^B_r(\omega)\|_S \right) \int_0^t \|\eta_r\|_{S^*} dr + C t\\
     &\le \|x\|_H^2 - \left( c- \delta \right) \int_0^t \|\eta_r\|_{S^*} dr + C t,
  \end{align*}
  where $\eta_r \in A(X(r,0;\omega)x)$ for a.e.\ $r \in [0,T]$ and $X=Y+W^B$, as before. The same argument can be applied for $u(t)x$ to obtain
     \[\int_{0}^{T} \|\tilde\eta_r\|_{S^*} dr + \int_{0}^{T} \|\eta_r\|_{S^*} dr \le C (1+ \|x\|_H^2),\]
  where $\tilde\eta_r \in A(u(t)x)$ and $\eta_r \in A(X(r,0;\omega)x)$ for a.e.\ $r \in [0,T]$ and all $\delta$ small enough. Using the monotonicity of $A$ we estimate the difference to the deterministic solution by 
  \begin{align*}
     \|Y(t,0;\omega)x-u(t)x\|_H^2  
     &=  \int_0^t \dualdel{S}{\eta_r-\tilde \eta_r}{Y(r,0;\omega)x-u(r)x} dr\\
     &\le \|\eta-\tilde \eta\|_{L^1([0,T];S^*)} \|W^B(\omega)\|_{L^\infty([0,T];S)} \\
     &\le C (1+ \|x\|_H^2) \delta,
  \end{align*}
  for $\omega \in \Omega^\delta$. By continuity in $x$ this inequality remains true for all $x \in H$. Hence
  \begin{align*}
    \|X(t,0;\omega)x-u(t)x\|_H^2 
    &= \|Y(t,0;\omega)x+W^B_t(\omega)-u(t)x\|_H^2 \\
    &\le C (1+ \|x\|_H^2) \delta,\quad \forall \omega \in \Omega^\delta.
  \end{align*}
  Choosing $\delta > 0$ small enough thus yields the claim.
\end{proof}

Next, we prove the asymptotic concentration of the average mass on bounded (compact resp.) sets.
Recall the definition of $Q^T(x,\cdot)$, right after \eqref{def:QT}.

\begin{lem}\label{ergodicity_lemma_1}
  Suppose that Hypotheses \ref{hyp:A}, \ref{hyp:noise_2}, \ref{hyp:extinction} hold. For each $\varepsilon > 0$ and each bounded set $A \subseteq H$ there exists a constant $R(\eps,\|A\|_H)>0$ such that the sublevel set $B:=\{\Theta\le R(\eps,\|A\|_H)\} \subseteq H$ satisfies
    \[ \inf_{x \in A} \liminf_{T \to \infty} Q^T(x,B) > 1- \varepsilon. \]
\end{lem}
\begin{proof}
Let $\eps>0$, $A\subset H$ be bounded and $x\in A$. For $R>0$, $K_R:=\{\Theta\le R\}$ is a measurable set. By It\={o}'s formula,
\[\frac{1}{T}\E\int_0^T\Theta(X(s,0;\cdot)x)\,ds\le C(\norm{x}^2_H+1),\quad\text{for $T\ge 1$.}\]
 Thus,
\begin{align*}
  Q^T(x,K_R)&=\frac{1}{T}\int_0^T P_s(x,K_R)\,ds \\
  &\ge\frac{1}{T}\int_0^T\left(1-\frac{\E\left[\Theta(X(s,0;\cdot)x)\right]}{R}\right)\,ds\ge 1-\frac{C}{R}(\norm{A}^2_H+1).
\end{align*}
Choosing $R(\eps,\|A\|_H)>\eps^{-1}C(\norm{A}^2_H+1)$ yields the claim with $B:=K_{R(\eps,\|A\|_H)} = \{x \in H|\ \Theta(x) \le R(\eps,\|A\|_H)\}$.
\end{proof}

We are now ready to prove a locally uniform lower bound for the average mass at $0$.

\begin{lem}
Suppose that Hypotheses \ref{hyp:A}, \ref{hyp:noise_2}, \ref{hyp:A_2}, \ref{hyp:extinction} hold. For each $\delta > 0$ and each bounded set $A \subseteq H$
   \[ \inf_{x \in A} \liminf_{T \to \infty} Q^T(x,B_\delta(0)) > 0. \]
\end{lem}
\begin{proof}
  Let $\delta > 0$, $A \subseteq H$ be bounded, $x \in A$ and $B=K_{R(\frac{1}{2},\|A\|_H)}$ be as in Lemma \ref{ergodicity_lemma_1}. By Hypothesis \ref{hyp:extinction}, there exists a $T_0$ corresponding to $B$ such that $\|u_{T_0}^z\|_H \le \frac{\delta}{2}$ for all $z \in B$. Using Lemma \ref{stoch_stable_lemma} yields
  \begin{align}\label{eqn:concentration_at_0}
    P_{T_0}(z,B_\delta(0)) 
    &=  \P(\|X^z_{T_0}\|_H \le \delta) \ge \P\left(\|X^z_{T_0}-u(T_0)z\|_H \le \frac{\delta}{2}\right) \ge \gamma_1 > 0,
  \end{align}
  where $\gamma_1=\gamma_1(T_0,\delta)$ is independent of $z \in B$. Thus
  \begin{align*}
     \liminf_{T \to \infty} Q^T(x,B_\delta(0))
     &= \liminf_{T \to \infty} \frac{1}{T} \int_0^T P_s(x,B_\delta(0))\,ds \\
     &= \liminf_{T \to \infty} \frac{1}{T} \int_0^T P_{s+T_0}(x,B_\delta(0))ds \\
     &= \liminf_{T \to \infty} \frac{1}{T} \int_0^T \int_H P_s(x,dz) P_{T_0}(z,B_\delta(0))\, ds \\
     &\ge \liminf_{T \to \infty} \frac{1}{T} \int_0^T \int_B P_s(x,dz) P_{T_0}(z,B_\delta(0))\, ds \\
     &\ge \liminf_{T \to \infty} \gamma_1 \frac{1}{T} \int_0^T \int_B P_s(x,dz)\, ds \\
     &\ge \gamma_1 \liminf_{T \to \infty}  Q^T\big(x,B \big) \ge \frac{\gamma_1}{2}  >  0,
  \end{align*}
  where $\gamma_1 = \gamma_1(T_0,\delta) = \gamma_1(\|A\|_H,\delta)$.
\end{proof}
We conclude that Theorem \cite[Theorem 2]{KPS} can be applied to yield the claim of Theorem \ref{thm:ergodicity}.

\section{Examples}\label{sect:ex}

In the following we present several examples of singular SPDE which can be treated by the general results discussed above. Let us first consider the case in which the drift is given as the subgradient of a convex function. Let $S \subseteq H \subseteq S^*$ be a Gelfand triple of Hilbert spaces and $T$, $T_n$, $J_n$ as in Section \ref{sec:det_case}.

\begin{prop}\label{prop:subgradient}
  Let $\phi: S \to \R$ be a lower semi-continuous, convex function such that $\inf_{u\in S}\phi(u)>-\infty$. Then $A := \partial\phi: S \to 2^{S^*}$ satisfies $\textup{(A1)}$, $\textup{(A4)}$, $\textup{(A5)}$.
  
  If
  \begin{equation}\label{eq:phi-quadratic-bound}
    \phi(u) \le C\left(1+\norm{u}_S^2\right),\quad \forall u \in S,
  \end{equation}
  for some $C> 0$, then $A$ satisfies $\textup{(A2)}$.

  If in addition, $\phi$ is non-expansive with respect to $J_n$, i.e.\ $\phi(J_n u) \le \phi(u)$ for all $u \in S$, $n\in\N$ then $A$ satisfies $\textup{(A3)}$.
\end{prop}
\begin{proof}
  By replacing $\phi$ with $\phi-\inf_{u\in S}\phi(u)$, we may assume $\phi\ge 0$. By \cite[Proposition 3.3]{Phel}, $\phi$ is continuous on $S$. (A4) is satisfied since $\phi$ is independent of time $t$.

  (A1): By \cite[Theorem 2.8, Proposition 1.7]{barbu2010nonlinear} $A$ is maximal monotone with nonempty values.
 
  (A5): As noted above, $\phi$ is continuous. Hence, there is $\delta >0$ such that $\phi(w) \le \phi(0)+1$ for all $\|w\|_S \le \delta$. For $v \in \partial\phi(u)$ we observe
    \begin{align*}
      \|v\|_{S^*} 
     &= \frac{1}{\delta} \sup_{w \in S,\ \|w\|_S = \delta} \dualdel{S}{v}{w} \\
     &= \frac{1}{\delta} \sup_{w \in S,\ \|w\|_S = \delta} \dualdel{S}{v}{w-u} + \frac{1}{\delta}\dualdel{S}{v}{u} \\
     &\le \frac{1}{\delta} \sup_{w \in S,\ \|w\|_S = \delta} \phi(w)-\phi(u) + \frac{1}{\delta}\dualdel{S}{v}{u} \\
     &\le \frac{1}{\delta} (\phi(0)+1) + \frac{1}{\delta}\dualdel{S}{v}{u}.
    \end{align*}

 (A2): Let $u\in S$, $v\in \partial\phi(u)$ and $\phi^\ast$ be the Legendre-Fenchel transform of $\phi$. By \eqref{eq:phi-quadratic-bound} it follows that
    \[\phi^\ast(v) \ge \frac{\norm{v}_{S^\ast}^2}{4C}-C.\quad \forall v \in S^*,\]
    compare e.g. with \cite[eq. 11(3)]{RockWets}.
  Hence by Young's equality,
  \begin{align*}
    \norm{v}_{S^\ast}^2 
    &\le 4C\left(\phi^\ast(v)+C\right)=4C\left(\dualdel{S}{v}{u}-\phi(u)+C\right)\\
    &\le 4C\left(\norm{v}_{S^\ast}\norm{u}_S+C\right),
  \end{align*}
  which implies (A2).

  (A3): For $u\in S$, $v \in \partial\phi(u)$
    $$ \dualdel{S}{v}{T_n u} = -n\dualdel{S}{v}{J_n u-u} \ge n\left(\phi(u) - \phi(J_n u)\right) \ge 0,$$
    for all $u \in S$ and $n\in\N$.
\end{proof}

\begin{prop}\label{prop:tendstozero}
  Let $\phi: H \to \R$ be a lower semicontinuous function with compact sublevel sets such that $\varphi_{|S}$ is an even, convex function satisfying $\phi(u)=\inf_{v\in H}\phi(v)$ iff $u = 0$. Then $A := \partial\phi$ satisfies Hypothesis \ref{hyp:extinction}.
\end{prop}
\begin{proof}
  We check the conditions of Remark \ref{rem:suff_L2}, (2). By the subgradient property we have
    $$\dualdel{S}{v}{u} \ge \phi(u)-\phi(0), \quad\forall v\in \partial\phi(u).$$
  Hence, we may choose $\Theta := \phi-\phi(0)$ in Hypothesis \ref{hyp:extinction}. Closedness of the sublevel sets on $H$ implies lower semi-continuity of $\Theta$ on $H$. Then \cite[Theorem 5]{Bruck} implies pointwise convergence to $0$.
\end{proof}

In the following let $(\Omega,\Fscr,\{\Fscr_t\}_{t \ge 0},\P)$ be a (not necessarily complete nor right-continuous) filtered probability space and $N:[0,T]\times\Omega\to S$ be an $\{\Fscr_t\}$-adapted process satisfying $N_\cdot(\omega)\in L^2([0,T];D(T^\frac{3}{2}))$ for all $\omega \in \Omega$ and strict stationarity, i.e.\ (N). Furthermore, let $(\Omega,\hat\Fscr,\{\hat\Fscr_t\}_{t \ge 0},\P)$ be a normal filtered probability space and $B: [0,T] \times \Omega \times S \to L_2(U,H)$ be an $\hat\Fscr_t$-progressively measurable map. The choice of the underlying Gelfand triple $S \subseteq H \subseteq S^*$ will be specified in each example.

\subsection{Stochastic generalized fast-diffusion equation}\label{sec:FDE}\

We adopt a framework inspired by R\"ockner and Wang \cite{RW}. Let $(E,\Bscr,\mu)$ be a finite measure space and
$(\Escr,D(\Escr))$ be a symmetric Dirichlet form on $L^2(\mu)$ with associated Dirichlet operator $(L,D(L))$ (cf.\ \cite{FOT}).
Assume that $L$ is strictly coercive, self-adjoint, positive-definite and possesses a compact resolvent. Then $D(\Escr)$ is a Hilbert space with norm $\norm{\cdot}_0:=\Escr^{1/2}(\cdot)$ and $D(\Escr)\subset L^2(\mu)$ is dense and compact.

We define the generalized fast diffusion operator in the Gelfand triple 
  \[S := L^2(\mu) \subseteq H:= D(\Escr)^\ast\subset S^\ast.\]
Let $\Psi: \R \to \R_+$ be an even, convex, continuous function with $\Psi(0) = 0$, subdifferential $\Phi = \partial\Psi: \R \to 2^\R$ and 
\begin{equation}\label{eqn:psi_growth}
     \Psi(r) \le C \left( |r|^2 + \mathbbm{1}_{\{\mu(E) < \infty\}} \right),\quad\forall r\in\R
\end{equation}
for some constant $C > 0$. Explicit examples are given by:
  \begin{enumerate}
   \item Fast diffusion equation: 
   \[\Phi_p(r) := \partial\left( s\mapsto\tfrac{1}{p}\abs{s}^{p}\right)(r) = |r|^{p-1}\sgn(r),\quad p\in [1,2].\] Note that we include the limit case $p=1$ for which 
      $$\Phi_1(r) = \begin{cases} \sgn(r), & r \ne 0 \\ [-1,1], & r =0. \end{cases}$$ 
   \item Plasma diffusion:
   \[\Phi_{\log}(r) := \partial\Big(s\mapsto(\abs{s}+1)\log(\abs{s}+1)-\abs{s}\Big)(r) = \log(|r|+1)\sgn(r).\]
  \end{enumerate}
  
We define
   $$\phi(u) := \int_E \Psi(u(\xi))\, d\mu(\xi)\in [0,+\infty],\quad u \in L^1(\mu)$$
and consider its restriction to $S$. Let $A := \partial\phi: S \to 2^{S^*}$. Taking \eqref{eqn:psi_growth} and $\Psi \ge 0$ into account,
we can apply Pratt's lemma \cite[Theorem 1]{pratt60} in order to see that
$\phi$ is continuous on $S$.

 As in \cite[Proposition 2.9]{Barb} we obtain
\begin{equation}\label{eqn:FDE_operator}
  A(u) = \left\{v \in S^* = L^2(\mu)\ \big|\ v(\xi) \in \Phi(u(\xi)),\ \text{a.e. } \xi \in E\right\}.  
\end{equation}

\begin{ex}
  Consider the stochastic generalized fast-diffusion equation
  \begin{equation}\label{eqn:FDE}\begin{split}
     dX_t&\in L\Phi(X_t)\,dt + \begin{cases} dN_t \\ B_t(X_t)\,dW_t \end{cases}, \quad t\in (0,T], \\
     X_0&=x.
  \end{split}\end{equation}  
  Then Theorem \ref{thm:unique_ex_stoch_add_pathw}, Corollary \ref{cor:gen_RDS}, Theorem \ref{thm:unique_ex_stoch_mult} and Theorem \ref{thm:unique_ex_stoch_mult_all_ic} apply, proving the unique existence of a (limit) solution to \eqref{eqn:FDE}.
\end{ex}
\begin{proof}
  By continuity, convexity and the growth condition for $\Psi$, $\phi:S \to \R_+$ is a convex continuous function. By Proposition \ref{prop:subgradient} it only remains to prove that $\phi$ is non-expansive with respect to $J_n$. 

  Recall that $-L$ equals the Riesz map of $S$ as we dualize over $H=D(\Escr)^\ast$. Since we have assumed $L$ to be a symmetric Dirichlet operator, it holds that 
      \[\norm{J_n u}_{L^\infty(\mu)}\le\norm{u}_{L^\infty(\mu)}, \quad\forall u\in L^\infty(\mu),\;\;n\in\N\]
  and
    \[J_n u \le J_n v, \quad\forall u,v\in L^2(\mu),\;\; 0 \le u \le v\;\;\mu\text{-a.e.},\;\;n\in\N.\]
  By an interpolation theorem due to Maligranda (cf.\ \cite[Theorem 3]{Mali}), we obtain
    $$ \phi(J_n u) \le \phi(u), \quad \forall u \in \tilde L^\Psi(\mu) \supset L^2(\mu)=S, $$
where $$\tilde L^\Psi (\mu):= \left\{u: E \to \R\;\Big\vert\; u\;\text{measurable},\; \phi(u) := \int_E \Psi(u)\ d\mu < \infty \right\}.$$
\end{proof}

In order to prove ergodicity of the corresponding Markovian semigroup we require one of the following stronger coercivity conditions
\begin{enumerate}
 \item[(C1)] For each $v \in L^2(\mu)$ with $v(\xi) \in \Phi(u(\xi))$ for a.e.\ $\xi \in E$ we have
  \begin{equation*}\label{eqn:FDE_coerc}
    \int_E vu \ d\mu \ge c\|u\|_H^p,  
  \end{equation*}
  for some $p \in [1,2)$ and some constant $c > 0$. 

  For example, this is satisfied in the case of stochastic fast diffusion equations on bounded domains, i.e.\ for $E = \Lambda \subseteq \R^d$ being a bounded, smooth domain, $\mu =dx$ being the Lebesgue measure, $L=\Delta$ being the Dirichlet Laplacian and $\Psi(r)=\tfrac{1}{p}\abs{r}^{p}$ with $\ p\in \left(1 \vee \frac{2d}{d+2},2 \right) \cup [d,2)$ (cf. also \cite{Liu4,LiuToe1}).
  \item[(C2)] The embedding $L^1(\mu)\subset H=D(\Escr)^\ast$ is compact, $\Psi(r) = 0$ iff $r = 0$ and 
      $$\lim_{r \to +\infty} \frac{\Psi(r)}{r} = +\infty.$$
    
    In particular, this is satisfied by the plasma diffusion in one space dimension. 
\end{enumerate}

\begin{ex}\label{ex:erg_FDE}
  Consider the generalized stochastic fast-diffusion equation,
    \begin{equation}\label{eqn:FDE_2}\begin{split}
     dX_t&\in L\Phi(X_t)\,dt + B\,dW_t,\quad t\in (0,T], \\
     X_0&=x
  \end{split}\end{equation}  
  and assume that \textup{(C1)} or \textup{(C2)} is satisfied. Then the associated Markovian semigroup $\{P_t\}$ is weak-$*$ mean ergodic. 
\end{ex}
\begin{proof}
  First assume (C1). We prove that Hypothesis \ref{hyp:extinction} is satisfied by Remark \ref{rem:suff_L2}, (1). By \eqref{eqn:FDE_operator} and (C1) we obtain
    $$ 2\dualdel{S}{y}{x} = \int_E y(\xi)x(\xi)\ d\mu(\xi) \ge c\|x\|_H^p. $$  
  Hence, Theorem \ref{thm:ergodicity} applies.

  Let us now assume (C2). We aim to prove that Hypothesis \ref{hyp:extinction} is satisfied by application of Proposition \ref{prop:tendstozero}.

  By Fatou's lemma $\phi$ is lower semi-continuous on $L^1(\mu)$. Compactness of the embedding $L^1(\mu)\subset H$ implies that $\phi$ has relatively compact sublevel sets in $H$. It remains to show closedness of the sublevel sets $\{\phi\le\alpha\}$ in $H$. Let $x_n\in\{\phi\le\alpha\}$, $n\in\N$, such that $\norm{x_n-x}_H \to 0$ for some $x\in H$. By de la Vall\'ee-Poussin's theorem, $\{x_n\}$ is uniformly integrable in $L^1(\mu)$. By the Dunford--Pettis theorem, there is $y\in L^1(\mu)$ and a subsequence $\{x_{n'}\}$ of $\{x_n\}$ such that $x_{n'}\rightharpoonup y$ weakly in $L^1(\mu)$. Since $\{\phi\le\alpha\}$ is convex and closed in $L^1(\mu)$, it is also weakly closed in $L^1(\mu)$ by Mazur's lemma. Hence $y\in\{\phi\le\alpha\}$. By weak continuity of the embedding $L^1(\mu)\subset H$, we conclude $y=x$.

  The claim follows now from Proposition \ref{prop:tendstozero}.
\end{proof}

As pointed out above, explicit examples include the fast diffusion equation for $\ p\in \left(1 \vee \frac{2d}{d+2},2 \right) \cup [d,2)$ and the generalized stochastic plasma-diffusion equation for $d = 1$. Note that this includes the limit case of the fast diffusion equation $(\Psi_1(r) = |r|)$ in one space dimension.

\subsubsection{Ergodicity for the stochastic plasma-diffusion in two space dimensions}\ 

In the situation of the stochastic plasma-diffusion in two space dimensions the embedding $L^1(\Lambda) \subseteq H = (H_0^1(\Lambda))^*$ does not hold and thus condition (C2) cannot be used to prove ergodicity anymore. In particular, we cannot expect compactness of the level sets of 
  $$\varphi(u) = \int_\Lambda \Psi(u(\xi))d\xi = \int_\Lambda (|u(\xi)|+1)\log(|u(\xi)|+1)-|u(\xi)| d\xi.$$
However, in this section we provide a Lyapunov function $\Theta$ with bounded sublevel sets $K_R:=\{\Theta\le R\}$ in $H$ such that the solutions to the deterministic equation converge to zero uniformly on $K_R$ (i.e.\ Hypothesis \ref{hyp:extinction} holds).

Let $L=\Delta$ be the Dirichlet Laplace operator on a bounded Lipschitz domain $\Lambda\subset\R^2$. Let $S=L^2(\Lambda)$, $H=H^{-1}(\Lambda)$. Consider the stochastic plasma-diffusion with additive noise,
  \begin{equation}\label{eqn:plasma-diff-2d}\begin{split}
     dX_t&= \Delta\left[\log(\abs{X_t}+1)\sgn(X_t)\right]\,dt + B\,dW_t,\quad t\in (0,T], \\
     X_0&=x.
  \end{split}\end{equation}  

Let $\Phi_{\log}(r):=\log(\abs{r}+1)\sgn(r)$. Observe that $\Phi_{\log}\in C^{1,1}(\R)$, i.e.
$\Phi_{\log}$ is continuously differentiable and Lipschitz. Indeed,
$\Phi'_{\log}(r)=\frac{1}{\abs{r}+1}\in (0,1]$ for $r\in\R$.

Consider the deterministic equation
  \begin{equation}\label{eqn:plasma-diff-2d-det}\begin{split}
     du_t&= \Delta\left[\log(\abs{u_t}+1)\sgn(u_t)\right]\,dt,\quad t\in (0,T], \\
     u_0&=x\in H^{-1}(\Lambda).
  \end{split}\end{equation}  

\begin{thm}\label{thm:loglyapunov}
Let $u$ be a weak solution to \eqref{eqn:plasma-diff-2d-det}. 
Let
\[\theta(r):=\Phi_{\log}(r)r=\abs{r}\log(\abs{r}+1).\]
Then
\[\Theta(v):=\int_\Lambda\theta(v)\,d\xi\in [0,+\infty],\quad v\in H^{-1}(\Lambda)\cap L^1(\Lambda),\]
is a Lyapunov function such that $u$ and $\Theta$ satisfy Hypothesis \ref{hyp:extinction}.

Moreover, the Markovian semigroup $\{P_t\}$ associated to \eqref{eqn:plasma-diff-2d} is weak-$*$ mean ergodic.
\end{thm}
We extend $\Theta$ by $+\infty$ to all of $H^{-1}(\Lambda)$. Lower semicontinuity of $\Theta: H \to [0,\infty]$ then follows from coercivity of $\theta$, the Dunford-Pettis theorem and weak lower semicontinuity of $\Theta: L^1(\Lambda) \to [0,\infty]$. In turn, weak lower semicontinuity of $\Theta: L^1(\Lambda) \to [0,\infty]$ follows from convexity and Fatou's Lemma.

For the proof of Theorem \ref{thm:loglyapunov},
we shall need the following form of the Trudinger-Moser inequality:
\begin{lem}[\cite{Moser70}]\label{lem:moser}
Let $\Lambda\subset\R^2$ be a bounded, open domain and let $u\in H^{1}_0(\Lambda)$. Then there exists a constant $c=c(2)>0$
such that
\[\int_\Lambda \exp(\abs{u})\,d\xi\le c\abs{\Lambda}\exp\left(\frac{1}{16\pi}\norm{\nabla u}_{L^2(\Lambda)}^2\right).\]
\end{lem}
\begin{proof}
See e.g.\ \cite[Section 2, Corollary 1, p. 414]{NaSeYo97}.
\end{proof}

\begin{lem}\label{lem:Moserbound}
Let $u\in H^{1}_0(\Lambda)$ with $\norm{\nabla u}_{L^2(\Lambda;\R^d)}\le 1$. Let
$\theta(r):=\abs{r}\log(\abs{r}+1)$ and $\theta^\ast$ the associated convex conjugate function.
Then there exists a constant $C>0$, depending only on $\abs{\Lambda}$, such that
\[\int_\Lambda\theta^\ast(u)\,d\xi\le C.\]
\end{lem}
\begin{proof}
Let $r > 0$. Then
\[\theta(r)=r\log(r+1)\ge r\log(r+1)-r\ge r\log(r)-r.\]
Taking the convex conjugate (see e.g.\ \cite[Section 3.3, p. 49]{borwein06}) yields
\[\theta^\ast(r)\le \exp(r),\quad\forall r\ge 0,\]
and hence by symmetry
\[\theta^\ast(r)\le \exp(\abs{r}),\quad\forall r\in\R.\]
The claim follows now by Lemma \ref{lem:moser}.
\end{proof}

\begin{proof}[Proof of Theorem \ref{thm:loglyapunov}]
  We observe that $\theta$ is twice continuously differentiable with
    \[\theta'(r)=\frac{r}{\abs{r}+1}+\Phi_{\log}(r), \quad \theta''(r)=\frac{\abs{r}+2}{(\abs{r}+1)^2}.\]
  Since $\theta''(r)\in (0,2]$ for $r\in\R$, we see that $\theta$ is convex and $\theta'$ is Lipschitz. Moreover, $\theta$ is an $N$-function, that is, $\theta$ is even, convex, continuous, non-decreasing, $\theta(r)=0$ iff $r=0$ and
    \[\lim_{r\to 0}\frac{\theta(r)}{r}=0,\quad\lim_{r\to +\infty}\frac{\theta(r)}{r}=+\infty.\]
  $\theta$ satisfies the following form of the so-called $\Delta_2$-condition
  \begin{equation}
      \label{Delta2}\theta(2r)\le 4\theta(r),\quad\forall r\in\R.
  \end{equation}
  Indeed,
  \[\theta(2r)=\abs{2r}\log(\abs{2r}+1)\le 2\abs{r}\log\left((\abs{r}+1)^2\right)=4\abs{r}\log(\abs{r}+1)= 4\theta(r).\]
  For initial data $u_0\in L^2(\Lambda)$, we get by the chain-rule ($0\le s\le t\le T$)
    \[\int_\Lambda\theta(u_t)\,d\xi-\int_\Lambda\theta(u_s)\,d\xi=\int_s^t\int_\Lambda\theta'(u_r)\Delta\Phi_{\log}(u_r)\,d\xi\,dr.\]
  Integrating by parts and the chain-rule yield
    \[\int_\Lambda\theta(u_t)\,d\xi-\int_\Lambda\theta(u_s)\,d\xi=-\int_s^t\int_\Lambda\theta''(u_r)\Phi_{\log}'(u_r)\abs{\nabla u_r}^2\,d\xi\,dr\le 0\]
  and $t\mapsto\int_\Lambda\theta(u_t)\,d\xi$ is non-increasing. These informal calculations can be made rigorous by standard approximation techniques. Taking into account that
  \[\norm{u_t}_{-1}^2=\norm{u_0}_{-1}^2-2\int_0^t\int_\Lambda\theta(u_r)\,d\xi\,dr,\]
  we get for $u_0\in L^2(\Lambda)$ that
  \[\int_\Lambda\theta(u_t)\,d\xi\le\frac{\norm{u_0}_{-1}^2}{2t}.\]

  Let $\eps>0$, and let $k\in\Z$ with $2^{-k}\le\eps\le 2^{-k+1}$. Then $\theta(\frac{s}{\eps})\le 4^k\theta(s)$ by \eqref{Delta2}.
  Then
  \[\int_\Lambda\theta\left(\frac{u_t}{\eps}\right)\,d\xi\le \frac{4^k \norm{u_0}_{-1}^2}{2t}\le \frac{1}{t}2^{2k-1} \norm{u_0}_{-1}^2\le\frac{2}{t}\eps^{-2}\norm{u_0}_{-1}^2.\]
  Therefore, for $\eps:=(2\norm{u_0}_{-1}^2)^{1/2}t^{-1/2}$,
  \[\int_\Lambda\theta\left(\frac{u_t}{\eps}\right)\,d\xi\le 1\]
  and hence
  \[\norm{u_t}_{(\theta)}\le\eps=\sqrt{\frac{2}{t}}\norm{u_0}_{-1},\]
  where
  \[\norm{v}_{(\theta)}:=\inf\left\{k>0\;\bigg\vert\;\int_\Lambda\theta\left(\frac{v}{k}\right)\,d\xi\le 1\right\}\]
  denotes the so-called \emph{gauge-norm}.

  Now, by \cite[Theorem 1.2.6, (i), p.\ 15]{RR2} and \cite[Theorem 1.2.8, (ii), p.\ 17]{RR2} and Lemma \ref{lem:Moserbound},

  \begin{align}\label{eqn:decay}
    \norm{u_t}_{-1}&=\sup\left\{\int_\Lambda u_t v\,d\xi\;\bigg\vert\; v\in H^{1}_0(\Lambda),\;\norm{\nabla v}_{L^2(\Lambda;\R^d)}\le 1\right\}\nonumber\\
    &\le 2\norm{u_t}_{(\theta)}\sup\left\{\left(1\vee\int_\Lambda\theta^\ast(v)\,d\xi\right)\;\bigg\vert\; v\in H^{1}_0(\Lambda),\;\norm{\nabla v}_{L^2(\Lambda;\R^d)}\le 1\right\}\\
    &\le 2\sqrt{2}(1\vee C)\frac{1}{\sqrt{t}}\norm{u_0}_{-1}.\nonumber
  \end{align}

  By density of the embedding $L^2(\Lambda)\subset H^{-1}(\Lambda)$ and continuity in the initial data we get \eqref{eqn:decay} for all $u_0\in H^{-1}(\Lambda)$. In particular, $u_t \to 0$ for $t \to \infty$ uniformly on bounded sets in $H$.

  To verify Hypothesis \ref{hyp:extinction} it remains to prove boundedness of the sublevel sets $K_R:=\{\Theta\le R\}$. By \cite[Theorem 1.2.8, (i)]{RR2}, $\norm{u}_{(\theta)}\le 1$ iff $\Theta(u)\le 1$ and by the same calculation as for \eqref{eqn:decay}
    $$\norm{u}_{-1} \le C\norm{u}_{(\theta)}, \quad\forall u \in H^{-1}(\Lambda).$$
  If $\Theta(u) \le 1$ then $\norm{u}_{(\theta)}\le 1$. If $\Theta(u) > 1$ then $\norm{u}_{(\theta)} > 1$ and by \cite[Theorem 1.2.8, (i)]{RR2} $\Theta(u) \ge \norm{u}_{(\theta)}$. Hence,
    $$\norm{u}_{-1} \le C (1+\Theta(u)), \quad\forall u \in H^{-1}(\Lambda).$$
  and $\Theta$ has bounded sublevel sets.

  The proof is completed by Theorem \ref{thm:ergodicity}.
\end{proof}

\subsection{The stochastic singular \texorpdfstring{$\Phi$}{Phi}-Laplacian equation}

\subsubsection{Dirichlet boundary conditions in a bounded domain}\label{dirichlet_section}\ 

Let $d\in\N$, $\Lambda\subset\R^d$ be a bounded domain with piecewise smooth boundary $\partial\Lambda$ of class $C^2$ and $-\Delta$ be 
the Dirichlet Laplacian on $\Lambda$. We define $\abs{\cdot}$ and $\lrbr{\cdot}{\cdot}$ to be the Euclidean norm and the inner-product on $\R^d$ respectively. Let
  $$S:=H^1_0(\Lambda) \subseteq H:=L^2(\Lambda) \subseteq S^*,$$
where we endow $S$ with the equivalent norm $\norm{u}_{S}:=\left(\int_\Lambda\abs{\nabla u}^2\,d\xi\right)^{1/2}$. 

Let $\Psi(x) = \tilde\Psi(|x|): \R^d \to \R_+$ be a radially symmetric function with $\tilde\Psi: \R \to [0,+\infty)$ being even, convex, continuous, non-decreasing and satisfying $\tilde{\Psi}(0)=0$. We further assume
\begin{equation}\label{eqn:psi_growth_2}
    \Psi(x) \le C(|x|^2+1),\quad\forall x\in\R^d
\end{equation}
for some constant $C > 0$. Let $\Phi := \partial\Psi: \R^d \to 2^{\R^d}$. Explicit examples are given by:
\begin{enumerate}
  \item Singular $p$-Laplacian:  $\Phi_p(x) := \partial\left(\tfrac{1}{p}\abs{\cdot}^{p}\right)(x) = |x|^{p-1}\Sgn(x),\ p\in [1,2],$ where 
  \[\operatorname{Sgn}(x):=\left\{\begin{aligned}&\dfrac{x}{\abs{x}},&&\;\;\text{if}\;\;x\in\R^d\setminus\{0\},\\
                 &\left\{y\in\R^d\;\vert\;\abs{y}\le 1\right\},&&\;\;\text{if}\;\;x=0.
                \end{aligned}\right.\]
  Note that we include the limit case $p=1$.
  \item Minimal surface flow: $$\Phi_{\text{m.s.f.}}(x):=\partial\left(\sqrt{1+\abs{\cdot}^2}\right)(x)=\frac{x}{\sqrt{1+\abs{x}^2}},\quad x\in\R^d.$$
  \item Plastic antiplanar shear deformation: $$\Phi_{\text{p.a.s.}}(x):=\partial\left(y\mapsto\begin{cases}\frac{1}{2}\abs{y}^2,&\;\;\text{if}\;\;\abs{y}\le 1,\\ \abs{y}-\frac{1}{2},&\;\;\text{if}\;\;\abs{y}> 1,\end{cases}\right)(x)=\begin{cases}x,&\;\;\text{if}\;\;\abs{x}\le 1,\\ \operatorname{Sgn}(x),&\;\;\text{if}\;\;\abs{x}> 1,\end{cases},\;\; x\in\R^d.$$
  \item Curve shortening flow:  
      $$\Phi_{\arctan}(r) := \partial\left( s\mapsto s\arctan(s)-\frac{1}{2}\log(s^2+1)\right)(r) = \arctan(r),\quad r \in \R.$$
\end{enumerate}
We define 
  \[\phi(u) := \int_\Lambda \Psi(\nabla u(\xi))\ d\xi,\quad u\in S\]
and $A := \partial\phi: S \to 2^{S^*}$. Since $\nabla: S=H_0^1(\Lambda) \to L^2(\Lambda;\R^d)$ is a linear, continuous operator, Pratt's lemma \cite[Theorem 1]{pratt60} (using \eqref{eqn:psi_growth_2} and $\Psi \ge 0$) implies that $\phi$ is continuous on $S$.

\begin{rem}
We would like to point out that we do not consider the $1$-Laplace operator on $H=L^2(\Lambda)$,
that is, the $L^2(\Lambda)$-subdifferential of the energy
\[\ol{\phi}(u):=\begin{cases}\norm{Du}(\Lambda)+\int_{\partial\Lambda}\abs{u^\Lambda}\,d\Hscr^{d-1},&\;\;\text{if}\;\;u\in BV(\Lambda)\cap L^2(\Lambda),\\
+\infty,&\;\;\text{otherwise.}\end{cases}\]
where $\norm{Du}(\Lambda)$ is the total variation of $u$ (measured in $\Lambda$), $u^\Lambda$
is the boundary trace of $u$ on $\partial\Lambda$ and $\Hscr^{d-1}$ is the $(d-1)$-dimensional Hausdorff
measure on $\partial\Lambda$. $\ol{\phi}$ is the closure of $u\mapsto\int_\Lambda\abs{\nabla u}\,dx$, $u\in W^{1,1}_0(\Lambda)$ in $L^2(\Lambda)$ (cf.\ \cite[Proposition 11.3.2]{ABM}, \cite[Section 2]{BR12}).

The $L^2(\Lambda)$ subdifferential $\partial\ol{\phi}$ was characterized in \cite[Proposition 4.23 (1)]{KaSchu}. As our framework relies on the triple $S\subset H\subset S^\ast$, for our purposes, it is sufficient to consider $\phi(u):=\int_\Lambda\abs{\nabla u}\,dx$ for $u\in H^1_0(\Lambda)$. Since $\varphi$ is continuous with respect to this stronger $H^1_0(\Lambda)$-norm the maximal monotonicity of $A:=\partial\varphi: S\to 2^{S^*}$ follows from standard results. Note $\phi=\ol{\phi}$ on $S=H_0^1(\Lambda)$.
\end{rem}

Since $\nabla: S=H_0^1(\Lambda) \to L^2(\Lambda;\R^d)$ has as adjoint $\nabla^* = (-\Delta)^{-1}\circ\ \div$, the chain-rule for subgradients (cf.\ \cite[Proposition 7.8]{Show}) implies $\partial\phi = \nabla^* \circ \partial\left(\int_\Lambda \Psi(\cdot)\ d\xi\right) \circ \nabla$, i.e.\ $\nabla^* y \in \partial\phi(x) \subseteq S^*$ iff $y \in L^2(\Lambda;\R^d)$ and $y(\xi) \in \Phi(\nabla x(\xi))$ for almost every $\xi \in \Lambda$. In this case
\begin{equation}\label{eqn:char_D_pLE}
  \dualdel{S}{\nabla^* y}{z} = \int_\Lambda \lrbr{y(\xi)}{ \nabla z}\ d\xi. 
\end{equation}
Hence $A$ is an $S$-realization of the singular $\Phi$-Laplace operator
  \[\text{``$A(u):=-\div \Phi(\nabla u)$''}.\]

\begin{ex}\label{ex:SpLE_Dirichlet}
  Consider the stochastic, singular $\Phi$-Laplace equation
  \begin{equation}\label{eqn:SpLE}\begin{split}
          dX_t&\in\div \Phi(\nabla X_t)\,dt+\begin{cases} dN_t \\ B_t(X_t)dW_t \end{cases}, \quad t\in (0,T], \\
          X_0&=x,
      \end{split}\end{equation}
  with Dirichlet boundary conditions. Then Theorem \ref{thm:unique_ex_stoch_add_pathw}, Corollary \ref{cor:gen_RDS}, Theorem \ref{thm:unique_ex_stoch_mult} and Theorem \ref{thm:unique_ex_stoch_mult_all_ic} apply, proving the unique existence of a (limit) solution to \eqref{eqn:SpLE}.
\end{ex}
\begin{proof}
  By continuity, convexity and the growth bound of $\Psi$, $\phi:S \to \R_+$ is a convex, continuous function. By Proposition \ref{prop:subgradient} it only remains to prove that $\phi$ is non-expansive with respect to $J_n$. This follows from Propositions \ref{prop:brezis} and \ref{prop:quadraticgrowthcontraction} in the appendix.
\end{proof}

In order to prove ergodicity of the associated Markovian semigroup we need to assume
\begin{enumerate}
 \item[(C3)] For some $p \in [1,2)$ and for all $\nabla^* v \in \partial\phi(u) \subseteq S^*$
   \begin{equation}\label{eqn:coerc_D_pLE}
      \int_\Lambda \lrbr{v(\xi)}{  \nabla u(\xi)}\ d\xi \ge c\|u\|_H^p.
   \end{equation}
  For example, this is the case for the stochastic singular $p$-Laplace equation, i.e.\ for $\Phi(x) := \abs{x}^{p-1}\Sgn(x)$ with $p\in \left[1\vee\tfrac{2d}{2+d},2\right)$ (cf.\ also \cite{LiuToe1}). For $p=1$, the result was conjectured in \cite{CiotToe}. 
\end{enumerate}

\begin{ex}\label{ex:SpLE_Dirichlet_erg}
  Consider the stochastic, singular $\Phi$-Laplace equation
  \begin{equation}\label{eqn:SpLE_2}\begin{split}
          dX_t&\in\div \Phi(\nabla X_t)\,dt+B\,dW_t, \quad t\in (0,T], \\
          X_0&=x,
      \end{split}\end{equation}
  with Dirichlet boundary conditions and assume \textup{(C3)}. Then the associated Markovian semigroup $\{P_t\}$ is weak-$*$ mean ergodic.
\end{ex}
\begin{proof}
   First assume (C3). We prove that Hypothesis \ref{hyp:extinction} is satisfied by Remark \ref{rem:suff_L2}, (1). Using \eqref{eqn:coerc_D_pLE} and \eqref{eqn:char_D_pLE} we observe
      $$2\dualdel{S}{\nabla^* y}{x} = \int_\Lambda \lrbr{y(\xi) }{\nabla x(\xi)}\, d\xi \ge c\|x\|_H^p,$$
  for all $y \in L^2(\Lambda;\R^d)$ with $y(\xi) \in \Phi(\nabla x(\xi))$ for a.e.\ $\xi \in \Lambda$. Thus, we may apply Theorem \ref{thm:ergodicity}.
\end{proof}

In particular, this applies to the stochastic total variation flow $(\Phi_1(x) = \operatorname{Sgn}(x))$ in one and two space dimensions.

\subsubsection{Neumann boundary conditions in a bounded domain}

Let $\Lambda\subset\R^d$ be a bounded, open domain with Lipschitz boundary $\partial\Lambda$. We either assume that $\Lambda$ is convex or $\partial\Lambda$ is $C^2$ and convex. By the Rellich--Kondrachov compact embedding theorem, $S:=H^1(\Lambda)\subset L^2(\Lambda)=:H$ compactly (and densely).
$S$ is normed by $\norm{u}_S^2:=\int_\Lambda\abs{\nabla u}^2\,d\xi+\int_\Lambda\abs{u}^2\,d\xi$.
Let $J_n$ be the resolvent of the Riesz map of $S$, which is the $H^1(\Lambda)$ realization
of $-\Delta+\Id$, where $\Delta$ denotes the Neumann Laplacian on $\Lambda$.

Let $\tilde{\Psi}:\R\to[0,+\infty)$ be an Orlicz function, i.e., $\tilde{\Psi}$ is even, convex, continuous, non-decreasing and satisfies $\tilde{\Psi}(0)=0$. Set $\Psi(z):=\tilde{\Psi}(\abs{z})$, $z\in\R^d$. Assume that for some constant $C>0$,
\begin{equation}\label{eq:Neumann_linear_growth}\Psi(z)\le C(1+\abs{z}^2)\quad\forall z\in\R^d.\end{equation}
Let $\Phi:=\partial\Psi$. An explicit example is given by the
singular $p$-Laplacian nonlinearity
\[\Psi_p(z):=\tfrac{1}{p}\abs{z}^p,\quad z\in\R^d,\]
where $p\in [1,2]$. The nonlinear diffusion operator ``$A(u):=-\operatorname{div}[\Phi(\nabla u)]$''
is defined rigorously as the subdifferential of the convex potential
\[\phi(u):=\int_\Lambda\Psi(\nabla u)\,d\xi,\quad u\in S.\]
Taking \eqref{eq:Neumann_linear_growth} and $\Psi\ge 0$ into account,
we can apply Pratt's lemma \cite[Theorem 1]{pratt60} in order to see that
$\phi$ is continuous on $H^1(\Lambda)$.

As in Section \ref{dirichlet_section}, $A:=\partial\phi$ admits the following variational characterization:
For $u\in S$,
$\nabla^\ast y\in A(u)$ iff $\abs{y}\in H$ and $y(\xi)\in\Phi(\nabla u(\xi))$ for a.e. $\xi\in\Lambda$.
In this case,
\[\dualdel{S}{\nabla^\ast y}{v}=\int_\Lambda\lrbr{y(\xi)}{\nabla v(\xi)}\,d\xi,\quad v\in S.\] 

\begin{ex}
  Consider the stochastic $\Phi$-Laplace equation with Neumann boundary conditions
  \begin{equation}\label{eqn:phiLE}\begin{split}
          dX_t&\in \div \Phi(\nabla X_t)  \,dt +\begin{cases} dN_t \\ B_t(X_t)dW_t \end{cases}, \quad t\in (0,T], \\
          X_0&=x\in L^2(\Lambda).
      \end{split}\end{equation}
  Then Theorem \ref{thm:unique_ex_stoch_add_pathw}, Corollary \ref{cor:gen_RDS}, Theorem \ref{thm:unique_ex_stoch_mult} and Theorem \ref{thm:unique_ex_stoch_mult_all_ic} apply, proving the unique existence of a (limit) solution to \eqref{eqn:phiLE}.
\end{ex}
\begin{proof}
By Proposition \ref{prop:subgradient}, we are done if we can show that $\phi$ is non-expansive
with respect to $\{J_n\}$. Let $\{J_n^0\}$ be the resolvent of $-\Delta$ and $\{P_t^0\}$ be the associated $L^2$ semigroup.
In our situation,
\cite[Eq.\ (1.1)]{WangYan13} holds with $K=0$. Therefore by \cite[Theorem 1.2]{WangYan13} or \cite[Proposition 2.5.1]{Wang},
we have the following gradient estimate for all $f\in C_b^1(\ol{\Lambda})$
\[\abs{\nabla P_t^0 f}\le P_t^0\abs{\nabla f},\quad\forall t\ge 0.\]
Let $\{P_t\}$ be the semigroup associated to $\{J_n\}$.
By the Trotter product formula \cite[Ch. VIII.8]{ReSi1},
$P_t=e^{-t}P_t^0$. Hence for all $n\in\N$,
\begin{multline*}\abs{\nabla J_n u}=\lrabs{\nabla \int_0^\infty e^{-t}P_{t/n}(u)\,dt}=\lrabs{\nabla \int_0^\infty e^{-t-\frac{t}{n}}P_{t/n}^0(u)\,dt}\\
\le \int_0^\infty e^{-t-\frac{t}{n}} P_{t/n}^0\abs{\nabla u}\,dt=\int_0^\infty e^{-t}P_{t/n}\abs{\nabla u}\,dt=J_n\abs{\nabla u},\quad\forall u\in C_b^1(\ol{\Lambda}). 
\end{multline*}
Since $\{P_t\}$ and hence $\{J_n\}$ is Markovian symmetric on $L^2(\Lambda)$, by \cite[Theorem 3]{Mali}
  \begin{equation*}\label{phisemigroup}
      \int_\Lambda \tilde\Psi(J_n u(\xi)) d\xi \le \int_\Lambda\tilde\Psi(u(\xi))d\xi,\quad\forall u\in L^2(\Lambda).
  \end{equation*}
  Combining these two inequalities we obtain
  \begin{align*}
     \phi(J_n u) 
     &= \int_\Lambda \tilde\Psi(|\nabla J_n u(\xi)|)\ d\xi \\
     &\le \int_\Lambda \tilde\Psi(J_n |\nabla u|(\xi))\ d\xi \\
     &\le \int_\Lambda \tilde\Psi(|\nabla u|(\xi))\ d\xi 
     = \phi(u),  \quad\forall u\in H^1(\Lambda),
  \end{align*}
  by density of $C_b^1(\ol{\Lambda})\subset H^1(\Lambda)$ and dominated convergence.
\end{proof}

\begin{rem}
Using the results from \cite{Wang,Wang05}, it is possible to deal with curved domains,
so that $\Lambda$ is replaced by a compact Riemannian manifold or a Riemannian manifold equipped with a weight.
\end{rem}

To prove ergodicity we again need to require a stronger coercivity condition, i.e.\ we assume
\begin{equation}\label{eqn:coerc_phiLE}
  \int_\Lambda \lrbr{\Phi(\nabla u(\xi))}{\nabla u(\xi)}\,d\xi \ge c \|u\|_H^p, 
\end{equation}
for some $c > 0$ and some $p \in [1,2)$. For example, by the Poincar\'e inequality, this is satisfied by the Neumann $p$-Laplacian with mean zero starting points for $p\in \left[ 1\vee\tfrac{2d}{2+d},2 \right)$, i.e., we replace
the energy space $H^1(\Lambda)$ above by the closed subspace $S:=H^1_{\textup{av}}(\Lambda):=\{u\in H^1(\Lambda)\;\vert\;\int_\Lambda u=0\}$ (cf.\  e.g.\ \cite{GigaKohn}). In this setup, we also need to
take $H:=L^2_{\textup{av}}(\Lambda):=\{u\in L^2(\Lambda)\;\vert\;\int_\Lambda u=0\}$.

By the same proof as for Example \ref{ex:SpLE_Dirichlet_erg} we obtain
\begin{ex}
  Consider the stochastic $\Phi$-Laplace equation with Neumann boundary conditions
  \begin{equation}\label{eqn:phiLE_add}\begin{split}
          dX_t&\in\operatorname{div}[\Phi(\nabla X_t) ]\,dt + B\,dW_t, \quad t\in (0,T], \\
          X_0&=x\in L^2_{\textup{av}}(\Lambda),
      \end{split}\end{equation}
  for $\Phi$ satisfying \eqref{eqn:coerc_phiLE}. Then the associated Markovian semigroup $\{P_t\}$ is weak-$*$ mean ergodic.
\end{ex}

\appendix

\section{A deterministic existence result}

Consider the following multi-valued Cauchy problem:
\begin{equation}\label{abstractcauchy}
\left\{
\begin{aligned}
\frac{d}{dt}u(t)+J(t,u(t))&\in F(t,u(t))\qquad\text{for a.e}\;\; t\in [0,T],\\
u(0)&=u_0.
\end{aligned}
\right.
\end{equation}
Here $J:[0,T]\times S\to S^\ast$ and $F:[0,T]\times S\to 2^{S^\ast}$.

\begin{defi}
A \emph{solution} of \eqref{abstractcauchy} is a function $u\in W^{1,2}(0,T)$
such that
\[\frac{d}{dt}u(t)+J(t,x(t))=f(t)\]
a.e. on $[0,T]$, $u(0)=u_0$ and $f\in L^2([0,T];S^\ast)$ such
that $f(t)\in F(t,u(t))$ for a.e. $t\in [0,T]$.
\end{defi}

Consider the following conditions on $J$ and $F$:
\begin{enumerate}[(J1)]
 \item $t\mapsto J(t,x)$ is measurable for all $x\in S$.
 \item $x\mapsto J(t,x)$ is demicontinuous and for almost all $t \in [0,T]$, $J(t,\cdot)$ satisfies: if $u_n\rightharpoonup u$ weakly in $S$ and
    \[\limsup_{n}\dualdel{S}{J(t,u_n)-J(t,u)}{u_n-u}\le 0,\]
   then $u_n\to u$ strongly in $S$.

  Note that the last property is satisfied whenever $J(t,\cdot)$ is strongly monotone, i.e.  
     \[\dualdel{S}{A(u)-A(v)}{u-v}\ge c\norm{u-v}^2_S\quad\forall u,v\in S,\ \text{for some } c>0.\]
 \item $\norm{J(t,x)}_{S^\ast}\le j_1(t)+c_1\norm{x}_S$ a.e. on $[0,T]$ for all $x\in S$ with $j_1\in L^2([0,T])$, $c_1>0$.
 \item $\dualdel{S}{J(t,x)}{x}\ge c_2\norm{x}_S^2-c_2'\norm{x}_S-j_2(t)$ a.e. on $[0,T]$ for all
$x\in S$ with $c_2,c_2'>0$ and $j_2\in L^1([0,T])$.
\end{enumerate}
\begin{enumerate}[(F1)]
 \item The values of $F$ are non-empty, closed and convex.
 \item $t\mapsto F(t,x)$ is measurable for all $x\in S$.
 \item $\operatorname{Gr} F(t,\cdot)$ is sequentially closed in $S\times S^\ast_w$ for almost all $t\in [0,T]$.
 \item $\norm{F(t,x)}_{S^\ast}\le f_1(t)+c_3\norm{x}_S$ for a.e. $t\in [0,T]$ with $f_1\in L^2([0,T])$, $c_3>0$.
 \item $\dualdel{S}{y}{x}\le \gamma(t)$ for a.e. $t\in [0,T]$, all $x\in S$, $y\in F(t,x)$ and some $\gamma\in L^1([0,T])_{+}$.
\end{enumerate}
Here ``$\operatorname{Gr}$'' denotes the graph of a multi-valued map,
\[L^1([0,T])_{+}:=\{f\in L^1([0,T])\;\vert\;f\ge 0\;\text{a.e. on}\;[0,T]\},\]
and ``$S_w^\ast$'' denotes the weak
Hilbert topology of $S^\ast$.

The following theorem can be found in \cite[Theorem I.2.40]{HuPapa2} or \cite[Theorem 3]{PapaPapaYanna}.
See \cite{BianWebb} for similar results.
\begin{thm}[Hu--Papageorgiou--Papalini--Yannakakis]\label{thm:Hupapa}
Suppose that $J$ and $F$ satisfy \textup{(J1)}--\textup{(J4)} and \textup{(F1)}--\textup{(F5)} resp.
Then the set of solutions to inclusion \eqref{abstractcauchy} for initial
point $u_0\in H$ is nonempty, weakly compact in $W^{1,2}(0,T)$ and compact
in $C([0,T];H)$.
\end{thm}

\section{Maximal monotone operators depending on a parameter}

Let $S$ be a separable Hilbert space, $(E,\Bscr,\mu)$ be a $\sigma$-finite measure space and $A:E \times S\to 2^{S^\ast}$ satisfy:
\begin{enumerate}
\item[(MM)] For $\mu$-a.a. $u \in E$ the map $x\mapsto A(u,x)$ is maximal monotone with nonempty values.
\end{enumerate}

\begin{defi}
  A map $A:S\to 2^{S^\ast}$ with non-empty values is called \emph{strongly-to-weakly upper semi-continuous} if for each $x\in S$ and for each weakly open set $V$ in $S^\ast$ such that $A(x)\subset V$ and for all $\{x_n\}\subset S$ with $x_n\to x$ strongly, we have that $A(x_n)\subset V$ for large $n$.

  A map $A:S\to 2^{S^\ast}$ with non-empty values is called \emph{weakly upper hemicontinuous} if for each $x,y\in S$, $\lambda\in [0,1]$ and for each weakly open set $V$ in $S^\ast$ such that $A(x+\lambda y)\subset V$ and for all $\{\lambda_n\}\subset [0,1]$ with $\lambda_n\to\lambda$, we have that $A(x+\lambda_n y)\subset V$ for large $n$.
\end{defi}

A map $A:S\to 2^{S^\ast}$ with non-empty values is weakly upper hemicontinuous whenever it is strongly-to-weakly upper semi-continuous.

Let $\tilde{A}$ be the mapping defined by:
\begin{align*}
  \tilde{A}&:     L^2(E,\Bscr,\mu;S)\to 2^{L^2(E,\Bscr,\mu;S^*)},\\
  \tilde{A}&:     x\mapsto\{y \in L^2(E,\Bscr,\mu;S^*)\;|\; y(u) \in A(u,x(u)),\ \mu\text{-a.e.}\; u \in E \}
\end{align*}
Note that $\tilde A (x)$ might be empty for some (or all) $x\in L^2(E,\Bscr,\mu;S)$.

\begin{prop}\label{prop:graphclosed}
  Suppose that \textup{(MM)} holds. Then the graph of $\tilde{A}$ is sequentially closed in $L^2(E,\Bscr,\mu;S)\times L^2_w(E,\Bscr,\mu;S^\ast)$.
\end{prop}
\begin{proof}
  Our proof is inspired by \cite[proof of Proposition (3.4)]{AttDam}.
  Let $x_n\to x$ in $L^2(E,\Bscr,\mu;S)$, such that $[x_n,y_n]\in\operatorname{Gr}\tilde{A}$ exists for each $n\in\N$ and assume that $y_n\rightharpoonup\eta$ weakly in $L^2(E,\Bscr,\mu;S^\ast)$ for some $\eta\in L^2(E,\Bscr,\mu;S^\ast)$. Let us extract a subsequence $\{x_{n_k}\}$ of $\{x_n\}$ such that $x_{n_k}(u)\to x(u)$ for $\mu$-a.e.\ $u\in E$.
  By the Banach-Saks theorem,
    \[w_{n}:=\frac{1}{n}\sum_{i=1}^{n} y_i\]
  converges to $\eta$ strongly in $L^2(E,\Bscr,\mu;S^\ast)$ (and hence $\{w_{n_k}\}$, too). By extracting another subsequence, if necessary, there is a $\mu$-null set $N\in\Bscr$ such that for all $u \in E \setminus N$
    \[x_{n_k}(u)\to x(u)\text{\;\;strongly in $S$},\quad w_{n_k}(u)\to\eta(u)\text{\;strongly in $S^\ast$}\]
  and $A(u,\cdot)$ is maximal monotone. By \cite[Ch. 3, Theorem 1.28]{HuPapa1}, $A(u,\cdot)$ is strongly-to-weakly upper semi-continuous. Let $u \in E \setminus N$ and $V$ be any weakly open neighborhood of $A(u,x(u))$. Then $y_{n_k}(u)\in V$ for large $k$ and hence $w_{n_k}(u)\in\textup{co}V$. As a result, $\eta(u)=\lim_k w_{n_k}(u)\in\ol{\textup{co}V}$. Since $V$ was chosen arbitrarily and $A(u,x(u))$ is a weakly closed, convex set (see e.g.\ \cite[Ch. 3, Proposition 1.14]{HuPapa1}), we conclude
    \[\eta(u) \in \bigcap_{\substack{V\text{\;weakly open}\\V\supset A(u,x(u))}}\ol{\textup{co}V}=A(u,x(u)).\]
  Since we can repeat the argument for all $u \in E\setminus N$, we get that $[x,\eta]\in\operatorname{Gr}\tilde{A}$.
\end{proof}

\begin{lem}\label{lem:hemi}
  Suppose that $A:S\to 2^{S^\ast}$ is a map with non-empty values. Then the following statements are equivalent.
  \begin{enumerate}[(i)]
    \item $A$ is maximal monotone.
    \item $A$ is monotone, has weakly compact and convex values and is weakly upper hemicontinuous.
  \end{enumerate}
\end{lem}
\begin{proof}
  Follows from \cite[Propositions 2.6.4 and 2.6.5]{ChangZhang}.

\end{proof}

\begin{prop}\label{prop:max_monotone}
  Let $(E,\Bscr,\mu)$ be a complete $\sigma$-finite measure space. Suppose that $A:E \times S\to 2^{S^\ast}$ is a map with non-empty values such that
  \begin{enumerate}
    \item[\textup{(1)}] $u \mapsto A(u,x)$ is $\Bscr$-measurable for all $x \in S$.
    \item[\textup{(2)}] $x \mapsto A(u,x)$ is maximal monotone for all $u \in E$.
    \item[\textup{(3)}] $\sup_{y\in A(x,u)}\norm{y}_{S^\ast}\le f(u)+C\norm{x}_S$ for $\mu$-a.a.\ $u \in E$ for all $x\in S$ and some $C>0$, $f\in L^2(E,\Bscr,\mu)_{+}$.
  \end{enumerate}
  Then $\tilde{A}:L^2(E,\Bscr,\mu;S)\to 2^{L^2(E,\Bscr,\mu;S^*)}$ is a maximal monotone map with non-empty values.
\end{prop}

The above conditions mimic those in \cite[cond. H(A)]{Papa1}, except for an additional coercivity condition that is assumed therein.

\begin{proof}[Proof of Proposition \ref{prop:max_monotone}]
  Compare with \cite[Proof of Lemma 3.3]{Papa1}. By condition (2) and \cite[Ch. 3, Theorem 1.28]{HuPapa1} $x \mapsto A(u,x)$ is strongly-to-weakly upper semi-continuous for all $u \in E$. By condition (1) and \cite[Theorem 1, Theorem 2]{Zyg} the map $u \mapsto A(u,x(u))$ is $\Bscr$-measurable and non-empty valued for all $x \in L^2(E,\Bscr,\mu;S)$. Condition (3) implies that each measurable selection $y \in A(\cdot,x(\cdot))$ is contained in $L^2(E,\Bscr,\mu;S^\ast)$, thus $\tilde A(x)$ is non-empty valued. Because of condition (2), $\tilde{A}$ is monotone. We are left to prove that $\tilde A$ is weakly upper hemicontinuous and has weakly compact, convex values in $L^2(E,\Bscr,\mu;S^\ast)$.

  We note that by condition (2) and Lemma \ref{lem:hemi}, for all $u \in E$ every $x \in L^2(E,\Bscr,\mu;S)$, $A(u,x(u))$ is a non-empty, convex and weakly compact set in $2^{S^\ast}$. It is easy to see that $A(\cdot,x(\cdot))$ is a convex subset of $L^2(E,\Bscr,\mu;S^\ast)$. As in the proof of Proposition \ref{prop:graphclosed} we see that $A(\cdot,x(\cdot))$ is weakly closed in $L^2(E,\Bscr,\mu;S^\ast)$. Weak compactness follows from condition (3) and the Banach-Alaoglu theorem.

  Now, let $x,v \in L^2(E,\Bscr,\mu;S)$, $\lambda\in [0,1]$, $V$ be a weakly open set such that $\tilde{A}(x+\lambda v) \subset V$ and $\{\lambda_n\}\subset [0,1]$ such that $\lambda_n\to\lambda$. Clearly, $x+\lambda_n v \to x+\lambda v$ strongly in $L^2(E,\Bscr,\mu;S)$. Suppose that ``$\tilde A(x+\lambda_n v)\subset V$ for large $n$'' is not valid. Then there exists a subsequence $y_{n_k} \in \tilde{A}(x+\lambda_{n_k} v)$ such that $y_{n_k} \not\in V$. Passing to a further subsequence, by condition (3) and the Banach-Alaoglu theorem, we may assume that $y_{n_k} \rightharpoonup \eta$ weakly for some $\eta\in L^2(E,\Bscr,\mu;S^\ast)$. By Proposition \ref{prop:graphclosed}, $\eta\in \tilde{A}(x+\lambda v)\subset V$ which gives a contradiction, since $V$ was assumed to be weakly open. Hence $\tilde A$ is weakly upper hemicontinuous and the proof is completed by Lemma \ref{lem:hemi}.
\end{proof}

\section{Comparison of different kinds of solutions}\label{SVIapp}

In \cite{BDPR,BDPR12,BR12} another approach for multi-valued singular stochastic evolution
equations was developed by introducing solutions in the sense of \emph{stochastic (evolution) variational inequalities} (SVI). In the following, we shall compare our notion of a solution in the sense of
Definitions  \ref{def:soln_mult}, \ref{def:stoch_soln_limit} with such SVI solutions (see also \cite[Remark 3.4]{BR12}).

Assume that there exists a proper, l.s.c., convex function $\phi:H\to [0,+\infty]$ such that
$\phi_S:=\phi\rs{S}$ is finite-valued on $S$ and admits a subdifferential $\partial\phi_S:S\to 2^{S^\ast}$
with $\partial\phi_S=A$. Assume also that
\begin{equation}\label{eq:quadradicgrowthinS}\phi_S(x)\le C(1+\norm{x}_S^2)\end{equation}
for some constant $C\ge 0$.

Let  $(\Omega,\hat\Fscr,\{\hat\Fscr_t\}_{t\ge 0},\P)$ be a normal filtered probability space.
Consider the following SPDE in $H$
\begin{equation}
\begin{split}\label{eq:SVI-SPDE}
dX_t+\partial\phi(X_t)\,dt&\ni B_t(X_t)\,dW_t, \quad t>0,\\
X_0&=x\in H,\end{split}\end{equation}
here, $W_t$ is a cylindrical Wiener process in $U$, where $U$ is some
separable Hilbert space, and $B:[0,T]\times\Omega\times H\to L_2(U,H)$ is assumed
to be $\{\hat\Fscr_t\}$-progressively measurable.

\begin{defi}\label{def:varn_soln}
Let $0<T<\infty$ and let $x\in H$. A stochastic process $X:[0,T]\times\Omega\to H$
is said to be an \emph{SVI solution} to \eqref{eq:SVI-SPDE} if the following conditions
hold.
\begin{enumerate}[(i)]
\item $X$ is $\{\hat\Fscr_t\}$-adapted, $X$ has $\P$-a.s.\ continuous sample paths in $H$ and $X_0=x$.
\item $X\in L^2([0,T]\times\Omega;H)$, $\phi(X)\in L^1([0,T]\times\Omega)$.
\item For all  $\{\hat\Fscr_t\}$-progressively measurable processes $G\in L^2([0,T]\times\Omega;H)$
and all $\{\hat\Fscr_t\}$-adapted $H$-valued processes $Z$ with $\P$-a.s.\ continuous sample paths
such that $Z\in L^2([0,T]\times\Omega;S)$ and $Z$ solving the equation
\[Z_t-Z_0+\int_0^t G_s\,ds=\int_0^t B_s(Z_s)\,dW_s,\quad t\in [0,T],\]
we have
\begin{equation}\begin{split}\label{eqn:var_ineq}
&\frac{1}{2}\E\norm{X_t-Z_t}_H^2+\E\int_0^t\phi(X_\tau)\,d\tau\\
&\le\frac{1}{2}\E\norm{x-Z_0}_H^2+\E\int_0^t\phi(Z_\tau)\,d\tau
+\E\int_0^t(X_\tau-Z_\tau,G_\tau)_H\,d\tau\\
&+\frac{1}{2}\E\int_0^t\norm{B_\tau(X_\tau)-B_\tau(Z_\tau)}_{L_2(U,H)}^2\,d\tau,\quad t\in [0,T].
\end{split}\end{equation}
There is at least one such pair $(G,Z)$.
\end{enumerate}
\end{defi}

Note that for additive noise, i.e.\ when $B$ does not depend on the position of $X$ in $H$,
the last term in the variational inequality introduced above vanishes.

\begin{rem}
If $B$ satisfies Hypothesis \ref{hyp:noise} and Hypothesis \ref{hyp:noise_2} then there exists at least one pair $(G,Z)$ satisfying the requirements of (iii) in the above Definition.
This follows (with $G\equiv 0$) e.g.\ from Theorem \ref{thm:unique_ex_stoch_mult}.
\end{rem}

\begin{prop}
Let $X$ be a solution (resp.\ limit solution) to \eqref{eq:SVI-SPDE} in the sense of Definition \ref{def:soln_mult}  (Definition \ref{def:stoch_soln_limit} resp.). Assume that $A$ satisfies Hypothesis \ref{hyp:A} and $B$ satisfies Hypothesis \ref{hyp:noise} and Hypothesis \ref{hyp:noise_2}. Then $X$ is also an SVI solution
to \eqref{eq:SVI-SPDE}.
\end{prop}
\begin{proof}
We first consider the case of $X$ being a solution in the sense of Definition \ref{def:soln_mult}. Then Definition \ref{def:varn_soln}, (i) is satisfied by assumption. Since $A$ satisfies Hypothesis \ref{hyp:A} also  Definition \ref{def:varn_soln}, (ii) is satisfied.
Let $(G,Z)$ be as in Definition \ref{def:varn_soln}. Using It\={o}'s formula \cite[Theorem 4.2.5]{PrRoe} for the $S^\ast$-valued process
\[X_t-Z_t=x-Z_0+\int_0^t (G_\tau-\eta_\tau)\,d\tau+\int_0^t (B_\tau(X_\tau)-B_\tau(Z_\tau))\,dW_\tau\]
and taking expectation yields
\begin{align*}
&\frac{1}{2}\E\norm{X_t-Z_t}^2_H+\E\int_0^t\dualdel{S}{\eta_\tau}{X_\tau-Z_\tau}\,d\tau\\
&=\frac{1}{2}\E\norm{x-Z_0}^2_H+\frac{1}{2}\E\int_0^t\norm{B_\tau(X_\tau)-B_\tau(Z_\tau)}_{L_2(U,H)}^2\,d\tau\\
&+\E\int_0^t(X_\tau-Z_\tau,G_\tau)_H\,d\tau,\quad t\in [0,T].
\end{align*}
By the subgradient property,
\begin{align*}
&\frac{1}{2}\E\norm{X_t-Z_t}_H^2+\E\int_0^t\phi_S(X_\tau)\,d\tau\\
&\le\frac{1}{2}\E\norm{x-Z_0}_H^2+\E\int_0^t\phi_S(Z_\tau)\,d\tau+\E\int_0^t(X_\tau-Z_\tau,G_\tau)_H\,d\tau\\
&+\frac{1}{2}\E\int_0^t\norm{B_\tau(X_\tau)-B_\tau(Z_\tau)}_{L_2(U,H)}^2\,d\tau,\quad t\in [0,T].
\end{align*}
Hence, $X$ is an SVI solution in the sense of Definition \ref{def:varn_soln}.

Let now $X$ be a limit solution in the sense of Definition \ref{def:stoch_soln_limit}. Again, Definition \ref{def:varn_soln}, (i) is satisfied by assumption. Let $x_n\in S$ with $\norm{x_n-x}_H\to 0$ as $n\to\infty$ and let $X^n$ be a solution to \eqref{eq:SVI-SPDE} in the sense of Definition \ref{def:soln_mult} starting in $x_n$. By Definition \ref{def:stoch_soln_limit} we have
   $$X^n \to X,\quad \text{in } L^2(\Omega;C([0,T];H)).$$
We have already seen that $X^n$ is an SVI solution and thus satisfies \eqref{eqn:var_ineq}. Taking the limit $n\to\infty$ in \eqref{eqn:var_ineq} and using  l.s.c.\ of $\varphi$ on $H$ we see that Definition \ref{def:varn_soln}, (iii) is satisfied by $X$. Definition \ref{def:varn_soln}, (ii)  then follows by choosing one such $(G,Z)$ which implies $\varphi(X) \in L^1([0,T]\times\Omega)$. 
\end{proof}

\begin{rem}\label{rem:compare_solutions}
\begin{enumerate}
 \item  Assume that $A$ satisfies Hypothesis \ref{hyp:A} and $B$ satisfies Hypothesis \ref{hyp:noise} and Hypothesis \ref{hyp:noise_2}. If uniqueness of SVI solutions to \eqref{eq:SVI-SPDE} can be proven then the unique limit solution $X$ coincides with the unique SVI solution.
 \item Uniqueness for SVI solutions has been verified for the stochastic $1$-Laplace equation with additive (cf.\ \cite{BDPR,BDPR12}) and linear multiplicative noise (cf.\ \cite{BR12}). Hypothesis \ref{hyp:A} has been proven for the $1$-Laplacian in Section \ref{dirichlet_section}. For additive noise, \cite[condition (1.3)]{BDPR} implies Hypothesis \ref{hyp:noise_2} while Hypothesis \ref{hyp:noise} is trivially satisfied. For linear multiplicative noise, \cite[conditions (H1) \& (H2)]{BR12} and the structure of the noise imply Hypothesis \ref{hyp:noise_2} and Hypothesis \ref{hyp:noise} (cf.\ \cite[Remark 2.1]{BR12} for details).
\end{enumerate}
\end{rem}

\section{\texorpdfstring{$W^{1,1}_0$}{W\_0\^\{1,1\} }-dissipativity of the Dirichlet Laplacian}

The main result \ref{prop:brezis} of this appendix is credited to H. Br\'ezis. For a detailed proof of
the results we refer to \cite[Section 8, Appendix 1]{BR12}.

\begin{prop}\label{prop:brezis}
Let $\Lambda$ be a bounded, convex domain of $\R^d$, $d\ge 1$, with piecewise smooth
boundary $\partial\Lambda$ of class $C^2$. Let $J_n=(1-\frac{\Delta}{n})^{-1}$, $n\in\N$, be the resolvent
of the Dirichlet Laplacian $(-\Delta,D(-\Delta))$, where $D(-\Delta)=H^{1}_0(\Lambda)\cap H^2(\Lambda)$. Then
\[\int_\Lambda\abs{\nabla J_n(u)}\,d\xi\le\int_\Lambda\abs{\nabla u}\,d\xi,\quad\forall u\in W^{1,1}_0(\Lambda),\;n\in\N.\]
\end{prop}
\begin{proof}
See \cite[Proposition 8.1, Remark 8.1]{BR12}.
\end{proof}

\begin{prop}\label{prop:quadraticgrowthcontraction}
Let $\Lambda$ be a bounded, convex domain of $\R^d$, $d\ge 1$, with piecewise smooth
boundary $\partial\Lambda$ of class $C^2$. Let $J_n=(1-\frac{\Delta}{n})^{-1}$, $n\in\N$, be the resolvent
of the Dirichlet Laplacian $(-\Delta,D(-\Delta))$, where $D(-\Delta)=H^{1}_0(\Lambda)\cap H^2(\Lambda)$.
Let $\tilde{\Psi}:\R\to [0,\infty)$ be a continuous, even and convex function of
at most quadratic growth such that $\tilde{\Psi}(0)=0$. Then
\[\int_\Lambda\tilde{\Psi}\left(\abs{\nabla J_n(u)}\right)\,d\xi\le\int_\Lambda\tilde{\Psi}\left(\abs{\nabla u}\right)\,d\xi,\quad\forall u\in H^{1}_0(\Lambda),\;n\in\N.\]
\end{prop}
\begin{proof}
\cite[Proposition 8.2]{BR12}.
\end{proof}

\providecommand{\bysame}{\leavevmode\hbox to3em{\hrulefill}\thinspace}
\providecommand{\MR}{\relax\ifhmode\unskip\space\fi MR }
\providecommand{\MRhref}[2]{%
  \href{http://www.ams.org/mathscinet-getitem?mr=#1}{#2}
}
\providecommand{\href}[2]{#2}

\end{document}